\documentclass[11pt,reqno]{amsart}
\usepackage{fullpage}
\usepackage{mathrsfs,amssymb,graphicx,verbatim,amsmath,amsfonts}
\usepackage{paralist}
\usepackage[breaklinks,pdfstartview=FitH]{hyperref}
\usepackage{upgreek}
\usepackage[usenames,dvipsnames,svgnames,table]{xcolor}
\usepackage{color}
\usepackage{hyperref}
\hypersetup{ocgcolorlinks=true,allcolors=testc}
\hypersetup{
     colorlinks   = true,
     citecolor    = black
}
\hypersetup{linkcolor=black}
\hypersetup{urlcolor=black}
\renewcommand{\leq}{\leqslant}
\renewcommand{\geq}{\geqslant}

\renewcommand{\gamma}{\upgamma}
\newcommand{\dd}{\mathrm{d}}

\newcommand{\e}{\varepsilon}
\newcommand{\R}{\mathbb R}

\newtheorem{theorem}{Theorem}
\newtheorem{lemma}[theorem]{Lemma}
\newtheorem{proposition}[theorem]{Proposition}

\newtheorem{corollary}[theorem]{Corollary}

\newtheorem{definition}[theorem]{Definition}

\theoremstyle{remark}
\newtheorem{remark}[theorem]{Remark}

\newtheorem{question}[theorem]{Question}

\newcommand{\C}{\mathbb C}

\renewcommand{\emptyset}{\varnothing}
\newcommand{\ent}{\mathrm{Ent}}
\newcommand{\proj}{\mathrm{Proj}}
\newcommand{\mb}{\mathbb}
\newcommand*\diff{\mathop{}\!\mathrm{d}}

\begin{document}

\title{Gaussian mixtures: entropy and geometric inequalities}
\author{Alexandros Eskenazis}
\address{$\mathrm{(A.~E.)}$ Mathematics Department\\ Princeton University\\ Fine Hall, Washington Road, Princeton, NJ 08544-1000, USA}
\email{ae3@math.princeton.edu}
\thanks{The authors were supported in part by the Simons Foundation. P.~N. was supported in part by NCN grant DEC-2012/05/B/ST1/00412.}

\author{Piotr Nayar}
\address{$\mathrm{(P.~N.)}$ Wharton Statistics Department\\ University of Pennsylvania\\ 3730 Walnut St, Philadelphia, PA 19104, USA}
\email{nayar@mimuw.edu.pl}

\author{Tomasz Tkocz}
\address{$\mathrm{(T.T.)}$ Mathematics Department\\ Princeton University\\ Fine Hall, Washington Road, Princeton, NJ 08544-1000, USA}
\email{ttkocz@princeton.edu}


\maketitle

\begin{abstract}
A symmetric random variable is called a {\it Gaussian mixture} if it has the same distribution as the product of two independent random variables, one being positive and the other a standard Gaussian random variable. Examples of Gaussian mixtures include random variables with densities proportional to $e^{-|t|^p}$ and symmetric $p$-stable random variables, where $p\in(0,2]$. We obtain various sharp moment and entropy comparison estimates for weighted sums of independent Gaussian mixtures and investigate extensions of the B-inequality and the Gaussian correlation inequality in the context of Gaussian mixtures.  We also obtain a correlation inequality for symmetric geodesically convex sets in the unit sphere equipped with the normalized surface area measure. We then apply these results to derive sharp constants in Khintchine inequalities for  vectors uniformly distributed on the unit balls with respect to $p$-norms and provide short proofs to new and old comparison estimates for geometric parameters of sections and projections of such balls.
\end{abstract}

\begin{section}{Introduction}
Gaussian random variables and processes have always been of central importance in probability theory and have numerous applications in various areas of mathematics. Recall that the measure $\gamma_n$ on $\R^n$ with density $\diff\gamma_n(x) = (2\pi)^{-n/2} e^{-\sum_{j=1}^n x_j^2/2} \diff x$ is called the standard Gaussian measure and a random vector distributed according to $\gamma_n$ is called a standard Gaussian random vector. A centered Gaussian measure on $\R^n$ is defined to be a linear image of standard Gaussian measure. In the past four decades intensive research has been devoted to geometric properties related to Gaussian measures (see, e.g., the survey \cite{Lat02}), which have provided indispensable tools for questions in convex geometry and the local theory of Banach spaces. In many cases, however, it still remains a challenging open problem to determine whether such properties are Gaussian per se or, in fact, more general.

The main purpose of the present article is to investigate properties of mixtures of Gaussian measures and demonstrate that they are of use to concrete geometric questions.

\begin{definition} \label{def:1}
A random variable $X$ is called a (centered) Gaussian mixture if there exists a positive random variable $Y$ and a standard Gaussian random variable $Z$, independent of $Y$, such that $X$ has the same distribution as the product $YZ$.
\end{definition}

For example, a random variable $X$ with density of the form
$$f(x) = \sum_{j=1}^m p_j \frac{1}{\sqrt{2\pi}\sigma_j} e^{-\frac{x^2}{2\sigma_j^2}},$$
where $p_j,\sigma_j>0$ are such that $\sum_{j=1}^m p_j = 1$, is a Gaussian mixture corresponding to the discrete random variable $Y$ with $\mb{P}(Y=\sigma_j)=p_j$. Finite weighted averages of non-centered Gaussian measures are ubiquitous in information theory and theoretical computer science (see, for instance, \cite{Das99}, \cite{AK01} for relevant results in learning theory) and are often referred in the literature as Gaussian mixtures. In this paper, we shall reserve this term for \emph{centered} Gaussian mixtures in the sense of Definition \ref{def:1}. Observe that Gaussian mixtures are necessarily symmetric and continuous. We shall now discuss a simple analytic characterization of Gaussian mixtures in terms of their probability density functions.

Recall that an infinitely differentiable function $g:(0,\infty)\to\R$ is called completely monotonic if $(-1)^n g^{(n)}(x) \geq 0$ for all $x>0$ and $n\geq0$, where for $n\geq1$ we denote by $g^{(n)}$ the $n$-th derivative of $g$ and $g^{(0)}=g$. A classical theorem of Bernstein (see, e.g., \cite{Fel71}) asserts that $g$ is completely monotonic if and only if it is the Laplace transform of some measure, i.e. there exists a non-negative Borel measure $\mu$ on $[0,\infty)$ such that
\begin{equation} \label{eq:bern}
f(x) = \int_0^\infty e^{-tx} \diff\mu(t), \ \ \mbox{for every } x>0.
\end{equation}
Bernstein's theorem implies the following equivalence.

\begin{theorem} \label{thm:charact}
A symmetric random variable $X$ with density $f$ is a Gaussian mixture if and only if the function $x\mapsto f(\sqrt{x})$ is completely monotonic for $x>0$. 
\end{theorem}

Theorem \ref{thm:charact} will be proven in Section \ref{sec:2}. It readily implies that for every $p\in(0,2]$ the random variable with density $c_p e^{-|x|^p}$ is a Gaussian mixture; we denote its law by $\mu_p$ and by $\mu_p^n =\mu_p^{\otimes n}$ the corresponding product measure. Another example of Gaussian mixtures are symmetric $p$-stable random variables, where $p\in(0,2]$ (see Lemma \ref{lemma:computations} in Section \ref{sec:2}). Recall that a symmetric $p$-stable random variable $X$ is a random variable whose characteristic function is $\mb{E} e^{itX} = e^{-c|t|^p}$, for $t\in\R$ and some $c>0$. Standard symmetric $p$-stable random variables correspond to $c=1$. In the consecutive subsections we shall describe our main results on Gaussian mixtures.


\subsection{Sharp Khintchine-type inequalities} The classical Khintchine inequality asserts that for every $p\in(0,\infty)$ there exist positive constants $A_p,B_p$ such that for every real numbers $a_1,\ldots,a_n$ we have
\begin{equation} \label{eq:radkhin}
A_p \Big(\sum_{i=1}^n a_i^2\Big)^{1/2} \leq \Big( \mb{E} \Big| \sum_{i=1}^n a_i \e_i \Big|^p \Big)^{1/p} \leq B_p \Big(\sum_{i=1}^n a_i^2\Big)^{1/2},
\end{equation}
where $\e_1,\ldots,\e_n\in\{-1,1\}$ are independent symmetric random signs. Whittle discovered the best constants in \eqref{eq:radkhin} for $p\geq3$ (see \cite{Whi60}), Szarek treated the case $p=1$ (see \cite{Sza76}) and finally Haagerup completed this line of research determining the optimal values of $A_p, B_p$ for any $p>0$ (see \cite{Haa81}).

Following Haagerup's results, sharp Khintchine inequalities for other random variables have also been investigated extensively (see, for example, \cite{LO95}, \cite{BC02}, \cite{NO12}, \cite{Koe14}). In particular, in \cite{LO95}, Lata\l a and Oleszkiewicz treated the case of i.i.d. random variables uniformly distributed on $[-1,1]$ and proved a comparison result in the sense of majorization that we shall now describe.

We say that a vector $a=(a_1,\ldots,a_n)$ is majorized by a vector $b=(b_1,\ldots,b_n)$, denoted $a\preceq b$, if the nonincreasing rearrangements $a_1^*\geq\ldots\geq a_n^*$ and $b_1^*\geq\ldots\geq b_n^*$ of the coordinates of $a$ and $b$, respectively, satisfy the inequalities
$$\sum_{j=1}^k a^*_j \leq \sum_{j=1}^k b_j^* \ \ \mbox{for each } k\in\{1,\ldots,n-1\} \ \ \mbox{and } \ \sum_{j=1}^n a_j = \sum_{j=1}^n b_j.$$
For a general reference on properties and applications of the majorization ordering see \cite{MO79}. For instance, every vector $(a_1,\ldots,a_n)$ with $a_i \geq 0$ and $\sum_{i=1}^n a_i=1$ satisfies
\begin{equation} \label{eq:order}
\Big( \frac{1}{n},\ldots,\frac{1}{n}\Big) \preceq (a_1,\ldots,a_n) \preceq (1,0,\ldots,0).
\end{equation}
A real-valued function which preserves (respectively reverses) the ordering $\preceq$ is called Schur convex (respectively Schur concave).

The main result of \cite{LO95} reads as follows. Let $U_1,\ldots, U_n$ be i.i.d. random variables, uniformly distributed on $[-1,1]$. For $p\geq2$ and $(a_1,\ldots,a_n), (b_1,\ldots,b_n) \in \R^n$ we have
\begin{equation} \label{eq:khinunif}
(a_1^2,\ldots,a_n^2) \preceq (b_1^2,\ldots,b_n^2) \ \Longrightarrow \ \mb{E}\Big| \sum_{i=1}^n a_i U_i \Big|^p \geq \mb{E}\Big| \sum_{i=1}^n b_iU_i\Big|^p
\end{equation}
and for $p\in[1,2)$ the second inequality is reversed. In particular, combining \eqref{eq:order} and \eqref{eq:khinunif}, for any $p\geq2$ and a unit vector $(a_1,\ldots,a_n)$ we get
\begin{equation} \label{eq:khinunif2}
\mb{E} |U_1|^p \leq \mb{E} \Big| \sum_{i=1}^n a_iU_i\Big|^p \leq \mb{E} \Big| \frac{U_1+\cdots+U_n}{\sqrt{n}} \Big|^p,
\end{equation}
whereas for $p\in[1,2)$ the reverse inequalities hold. Inequality \eqref{eq:khinunif2} along with the central limit theorem implies that the sharp constants in the Khintchine inequality
\begin{equation} \label{eq:khinunif3}
A_p \Big( \mb{E} \Big| \sum_{i=1}^n a_iU_i\Big|^2 \Big)^{1/2} \leq \Big( \mb{E} \Big| \sum_{i=1}^n a_iU_i\Big|^p \Big)^{1/p} \leq B_p \Big( \mb{E} \Big| \sum_{i=1}^n a_iU_i\Big|^2 \Big)^{1/2}
\end{equation}
are precisely
\begin{equation} \label{eq:unifconstants}
A_p = \begin{cases} \gamma_p,& p\in[1,2) \\ \frac{3^{1/2}}{(p+1)^{1/p}},& p\in [2,\infty) \end{cases} \ \ \ \mbox{and} \ \ \ B_p = \begin{cases} \frac{3^{1/2}}{(p+1)^{1/p}},& p\in[1,2) \\ \gamma_p,& p\in [2,\infty) \end{cases},
\end{equation}
where $\gamma_p = \sqrt{2} \left( \frac{\Gamma\left(\frac{p+1}{2}\right)}{\sqrt{\pi}}\right)^{1/p}$ is the $p$-th moment of a standard Gaussian random variable. 

Our main result for moments is an analogue of the Schur monotonicity statement \eqref{eq:khinunif} for Gaussian mixtures. Recall that for a random variable $Y$ and $p\neq0$ we denote by $\|Y\|_p = (\mb{E} |Y|^p)^{1/p}$ its $p$-th moment and $\|Y\|_0 = \exp(\mb{E}\log|Y|)$. Notice that since a standard Gaussian random variable $Z$ satisfies $\mb{E}|Z|^p = \infty$ for every $p\leq -1$, a moment comparison result for Gaussian mixtures can only make sense for $p$-th moments, where $p>-1$.

\begin{theorem} \label{thm:khin}
Let $X$ be a Gaussian mixture and $X_1,\ldots,X_n$ be independent copies of $X$. For two vectors $(a_1,\ldots,a_n), (b_1,\ldots,b_n)$ in $\R^n$ and $p\geq2$ we have
\begin{equation} \label{eq:mixtkhin}
(a_1^2,\ldots,a_n^2) \preceq (b_1^2,\ldots,b_n^2) \ \Longrightarrow \ \Big\| \sum_{i=1}^n a_iX_i\Big\|_p \leq \Big\|\sum_{i=1}^n b_iX_i\Big\|_p,
\end{equation}
whereas for $p\in (-1,2)$ the second inequality is reversed, provided that $\mb{E}|X|^p<\infty$.
\end{theorem}

The proof of Theorem \ref{thm:khin} and the straightforward derivation of sharp constants for the corresponding Khintchine inequalities (Corollary \ref{cor:khinconstants}) will be provided in Section \ref{sec:3}.

\begin{remark}
After the submission of this paper, we learned from C. Houdr\'e that Theorem \ref{thm:khin} had previously appeared in his joint work \cite[Proposition~2.6]{AH03} with R. Averkamp. We are grateful to C. Houdr\'e for providing us this reference.
\end{remark}

As an application we derive similar Schur monotonicity properties for vectors uniformly distributed on the unit ball of $\ell_q^n$ for $q\in(0,2]$, which were first considered by Barthe, Gu\' edon, Mendelson and Naor in \cite{BGMN05}. Recall that for a vector $x=(x_1,\ldots,x_n)\in\R^n$ and $q>0$ we denote $\|x\|_q = \big( \sum_{i=1}^n |x_i|^q\big)^{1/q}$ and $\|x\|_\infty = \max_{1\leq i \leq n} |x_i|$. We also write $\ell_q^n$ for the quasi-normed space $(\R^n,\|\cdot\|_q)$ and $B_ q^n = \{x\in\R^n: \ \|x\|_q\leq1\}$ for its closed unit ball. In \cite{BGMN05}, the authors discovered a representation for the uniform measure on $B_q^n$, relating it to the product measures $\mu_q^n$ defined after Theorem \ref{thm:charact}, and used it to determine the sharp constants in Khintchine inequalities on $B_q^n$ up to a constant factor. Using their representation along with Theorem \ref{thm:khin} we deduce the following comparison result.

\begin{corollary} \label{cor:ballkhin}
Fix $q\in(0,2]$ and let $X=(X_1,\ldots,X_n)$ be a random vector uniformly distributed on $B_q^n$. For two vectors $(a_1,\ldots,a_n),(b_1,\ldots,b_n)$ in $\R^n$ and $p\geq2$ we have
\begin{equation} \label{eq:ballkhin}
(a_1^2,\ldots,a_n^2) \preceq (b_1^2,\ldots,b_n^2) \ \Longrightarrow \ \Big\| \sum_{i=1}^n a_iX_i\Big\|_p \leq \Big\|\sum_{i=1}^n b_iX_i\Big\|_p,
\end{equation}
whereas for $p\in (-1,2)$ the second inequality is reversed.
\end{corollary}

The derivation of the sharp constants in the corresponding Khintchine inequality is postponed to Corollary \ref{cor:ballconstants}. Given Corollary \ref{cor:ballkhin} and the result of \cite{LO95}, which corresponds to the unit cube $B_\infty^n$, the following question seems natural.

\begin{question} \label{question:khin}
Let $X=(X_1,\ldots,X_n)$ be a random vector uniformly distributed on $B_q^n$ for some $q\in(2,\infty)$. What are the sharp constants in the Khintchine inequalities for $X$?
\end{question}

It will be evident from the proof of Corollary \ref{cor:ballkhin} that Question \ref{question:khin} is equivalent to finding the sharp Khintchine constants for $\mu_q^n$, where $q\in(2,\infty)$. We conjecture that there exists a Schur monotonicity result, identical to the one in \eqref{eq:khinunif}.

\medskip

\noindent {\it Remark added in proofs.} In the recent preprint \cite{ENT18}, Question 6 is resolved by different methods than those of the present article. In particular, it is shown there that the sharp constants in the Khinchine inequality for $\mu_q^n$, where $q\in(2,\infty)$, are the same as those for $\mu_\infty^n$ when $p\geq1$.


\subsection{Entropy comparison} For a random variable $X$ with density function $f:\R\to\R_+$ the Shannon entropy of $X$ is a fundamental quantity in information theory, defined as
$$\ent(X) = -\int_\R f(x) \log f(x) \diff x = \mathbb{E} [ -\log f(X)],$$
provided that the integral exists. Jensen's inequality yields that among random variables with a fixed variance, the Gaussian random variable maximizes the entropy. Moreover, Pinsker's inequality (see, e.g., \cite[Theorem~1.1]{GL10}) asserts that if a random variable $X$ has variance one and $G$ is a standard Gaussian random variable, then the entropy gap $\ent(G)-\ent(X)$ dominates the total variation distance between the laws of $X$ and $G$. Consequently, the entropy can be interpreted as a measure of \emph{closeness to Gaussianity}. The following question seems natural.

\begin{question} \label{question:entropy}
Fix $n\geq2$ and suppose that $X_1,\ldots,X_n$ are i.i.d. random variables with finite variance. For which unit vectors $(a_1,\ldots,a_n)$ is the entropy of $\sum_{i=1}^n a_iX_i$ maximized?
\end{question}
\noindent The constraint $\sum_{i=1}^n a_i^2 =1$ on $(a_1,\ldots,a_n)$ plainly fixes the variance of the weighted sum $\sum_{i=1}^n a_iX_i$ and the answer would give the corresponding \emph{most Gaussian} weights.

The first result concerning the entropy of weighted sums of i.i.d. random variables was the celebrated entropy power inequality, first stated by Shannon in \cite{SW49} and rigorously proven by Stam in \cite{Sta59}. An equivalent formulation of the Shannon-Stam inequality (see \cite{Lie78}) reads as follows. For every $\lambda\in[0,1]$ and independent random variables $X,Y$ we have
\begin{equation} \label{eq:linepi}
\ent(\sqrt{\lambda}X+\sqrt{1-\lambda}Y) \geq \lambda \ent(X)+(1-\lambda)\ent(Y),
\end{equation}
provided that all the entropies exist. It immediately follows from \eqref{eq:linepi} that if $X_1,\ldots,X_n$ are i.i.d. random variables with finite variance and $(a_1,\ldots,a_n)$ is a unit vector, then we have
\begin{equation} \label{eq:lb}
\ent\Big(\sum_{i=1}^n a_iX_i\Big) \geq \ent(X_1).
\end{equation}
In other words, the corresponding minimum in Question \ref{question:entropy} is achieved at the direction vectors $e_i$.

Moreover, a deep monotonicity result for Shannon entropy was obtained in the work of Artstein-Avidan, Ball, Barthe and Naor \cite{ABBN04}. The authors proved that for any random variable $X$ with finite variance and any $n\geq 1$ we have
\begin{equation} \label{eq:mon}
\ent\Big(\sum_{i=1}^n \frac{1}{\sqrt{n}} X_i \Big) \leq \ent\Big( \sum_{i=1}^{n+1}\frac{1}{\sqrt{n+1}}X_i \Big),
\end{equation}
where $X_1,X_2,\ldots$ are independent copies of $X$.

Given inequality \eqref{eq:mon}, a natural guess for Question \ref{question:entropy} would be that the vector $\big(\frac{1}{\sqrt{n}},\ldots,\frac{1}{\sqrt{n}}\big)$ is a maximizer for any $n\geq 2$ and for any square-integrable random variable $X$. However, this is not correct in general. In \cite[Proposition~2]{BNT16}, the authors showed that for a symmetric random variable $X$ uniformly distributed on the union of two intervals the Shannon entropy of the weighted sum $\sqrt{\lambda}X_1 + \sqrt{1-\lambda} X_2$ is not maximized at $\lambda = \frac{1}{2}$.

Nonetheless, for Gaussian mixtures it is possible to obtain the comparison for R\'enyi entropies which confirms the natural guess. Recall that for a random variable $X$ with density $f:\R\to\R_+$ and $\alpha>0$, $\alpha\neq1$, the R\'enyi entropy of order $\alpha$ of $X$ is defined as
$$h_\alpha(X) = \frac{1}{1-\alpha} \log \Big( \int_\R f^\alpha(x) \diff x\Big).$$
Note that if for some $\alpha > 1$ the integral of $f^\alpha$ is finite, then $h_\alpha(X)$ tends to $\ent(X)$ as $\alpha\to1^+$ (see \cite[Lemma~V.3]{BC15}), which we shall also denote by $h_1(X)$ for convenience.

\begin{theorem} \label{thm:entropycomp}
Let $X_1,\ldots,X_n$ be i.i.d. Gaussian mixtures and $\alpha\geq1$. Then for two vectors $(a_1,\ldots,a_n),(b_1,\ldots,b_n)$ in $\R^n$ we have
\begin{equation} \label{eq:renyicomp}
(a_1^2,\ldots,a_n^2) \preceq (b_1^2,\ldots,b_n^2) \ \Longrightarrow \ h_\alpha\Big(\sum_{i=1}^n a_i X_i \Big) \geq h_\alpha\Big(\sum_{i=1}^n b_iX_i\Big),
\end{equation}
provided that all the entropies are finite. In particular, for every unit vector $(a_1,\ldots,a_n)$
\begin{equation} \label{eq:entropycomp}
\ent(X_1) \leq \ent\Big(\sum_{i=1}^n a_i X_i\Big) \leq \ent\Big( \frac{X_1+\cdots+X_n}{\sqrt{n}}\Big).
\end{equation}
\end{theorem}

Extensions of inequality \eqref{eq:entropycomp}, even for the uniform measure on the cube, appear to be unknown.

\begin{question}
Let $U_1,\ldots,U_n$ be i.i.d. random variables, each uniformly distributed on $[-1,1]$. Is it correct that for every unit vector $(a_1,\ldots,a_n)$
\begin{equation} \label{eq:entcube}
\ent\Big(\sum_{i=1}^n a_iU_i \Big) \leq \ent\Big( \frac{U_1+\cdots+U_n}{\sqrt{n}}\Big) \ ?
\end{equation}
\end{question}

Geometrically, this would mean that, in the entropy sense, the \emph{most Gaussian} direction of the unit cube $B_\infty^n$ is the main diagonal.

We close this subsection with an intriguing question in the spirit of the well known fact that a Gaussian random variable has maximum entropy among all random variables with a specified variance. Note that Theorem \ref{thm:entropycomp} along with
$$(1,1,0,\ldots,0) \succeq \Big( 1, \frac{1}{2}, \frac{1}{2}, 0,\ldots,0\Big) \succeq \cdots \succeq \Big(1, \frac{1}{n}, \ldots, \frac{1}{n}\Big)$$
imply that for every i.i.d. Gaussian mixtures $X_1,X_2,\ldots$ the sequence $\ent\big( X_1+\frac{X_2+\cdots+X_{n+1}}{\sqrt{n}}\big)$, $n = 1, 2, \ldots$ is increasing and in particular
$$\ent(X_1+X_2) \leq \ent\Big( X_1 + \frac{X_2+\cdots+X_{n+1}}{\sqrt{n}}\Big).$$
Thus, the following result should not be surprising.

\begin{proposition} \label{prop:swap}
Let $X_1, X_2$ be independent Gaussian mixtures with finite variance. Then
\begin{equation} \label{eq:swap1}
\ent(X_1+X_2) \leq \ent(X_1+G),
\end{equation}
where $G$ is a Gaussian random variable independent of $X_1$ having the same variance as $X_2$.
\end{proposition}

We pose a question as to whether this is true in general, under the additional assumption that $X_1,X_2$ are identically distributed.

\begin{question} \label{quest:swap}
Let $X_1,X_2$ be i.i.d. continuous random variables with finite variance. Is it true that
\begin{equation} \label{eq:swap}
\ent(X_1+X_2) \leq \ent(X_1+G),
\end{equation}
where $G$ is a Gaussian random variable independent of $X_1$ having the same variance as $X_2$?
\end{question}

The preceding entropy comparison results will be proven in Section \ref{sec:3}.


\subsection{Geometric properties of Gaussian mixtures}

Recall that a function $\varphi:\R^n\to\R_+$ is called log-concave if $\varphi=e^{-V}$ for some convex function $V:\R^n\to(-\infty,\infty]$. A measure $\mu$ on $\R^n$ is called log-concave if for every Borel sets $A,B\subseteq\R^n$ and $\lambda\in(0,1)$ we have
\begin{equation} \label{eq:prekopa1}
\mu(\lambda A + (1-\lambda)B) \geq \mu(A)^\lambda \mu(B)^{1-\lambda}.
\end{equation}
A random vector is called log-concave if it is distributed according to a log-concave measure. Two important examples of log-concave measures on $\R^n$ are Gaussian measures and uniform measures supported on convex bodies. The geometry of log-concave measures, in analogy with the asymptotic theory of convex bodies, has been intensively studied and many major results are known (see, for example, the monograph \cite{AGM15}). The Gaussian measure, however, possesses many delicate properties which are either wrong or whose validity is still unknown for other log-concave measures. In what follows, we will explain how to extend, in the context of Gaussian mixtures, two such properties: the B-inequality, proven by Cordero-Erausquin, Fradelizi and Maurey in \cite{CFM04}, and the Gaussian correlation inequality, recently proven by Royen in \cite{Roy14}.

Choosing the sets $A,B$ in \eqref{eq:prekopa1} to be dilations of a fixed convex set $K\subseteq\R^n$ we deduce that for every $a,b>0$ and $\lambda\in(0,1)$
\begin{equation} \label{eq:prekopa2}
\mu\big( (\lambda a +(1-\lambda) b) K\big) \geq \mu(aK)^\lambda \mu(bK)^{1-\lambda}.
\end{equation}
The (weak) B-inequality provides a substantial strengthening of \eqref{eq:prekopa2} for Gaussian measure, under an additional symmetry assumption: for any origin symmetric convex set $K\subseteq\R^n$, $a,b>0$ and $\lambda\in(0,1)$
\begin{equation} \label{eq:Bineq1}
\gamma_n (a^\lambda b^{1-\lambda} K) \geq \gamma_n (aK)^\lambda \gamma_n(bK)^{1-\lambda},
\end{equation}
or, in other words, the function $t \mapsto \gamma_n(e^tK)$ is log-concave on $\R$. In fact, in \cite{CFM04} the following strong form of the above inequality was proven.

\begin{theorem} [strong B-inequality, \cite{CFM04}] \label{thm:Bineq}
Let $K$ be a symmetric convex set in $\R^n$. Then, the function
\begin{equation} \label{eq:Bthm}
\R^n \ni (t_1,\ldots,t_n) \longmapsto \gamma_n \big(\Delta(e^{t_1},\ldots,e^{t_n}) K\big)
\end{equation}
is log-concave on $\R^n$, where $\Delta(s_1,\ldots,s_n)$ is the diagonal $n\times n$ matrix with entries $s_1,\ldots,s_n$.
\end{theorem}

The authors also proved that the same conclusion holds for an arbitrary unconditional log-concave measure, provided that the convex set $K$ is unconditional as well (see \cite[Section~5]{CFM04} for further details). Furthermore, they asked whether the weak B-inequality holds for any symmetric log-concave measure and symmetric convex set $K$; this is currently known as the B-conjecture. We note that in \cite{Sar16}, Saroglou confirmed the B-conjecture on the plane (the case of uniform measures on convex planar sets had previously been treated in \cite{LBo14}). Our result in this direction is the following theorem.

\begin{theorem} \label{thm:Bineqmixt}
Let $X_1,\ldots,X_n$ be Gaussian mixtures such that $X_i$ has the same distribution as $Y_i Z_i$, where $Y_i$ is positive and $Z_i$ is a standard Gaussian random variable independent of $Y_i$. Denote by $\mu_i$ the law of $X_i$ and by $\mu$ the product measure $\mu_1\times \cdots \times \mu_n$. If, additionally, $\log Y_i$ is log-concave for each $i$, then for every symmetric convex set $K$ in $\R^n$ the function
\begin{equation} \label{eq:Bthmmixt}
\R^n \ni (t_1,\ldots,t_n) \longmapsto \mu(\Delta(e^{t_1},\ldots,e^{t_n})K)
\end{equation}
is log-concave on $\R^n$.
\end{theorem}

We do not know whether the additional assumption on the $Y_i$ can be omitted, but we verified (Corollary \ref{cor:Bineqexamples}) that both the measure with density proportional to $e^{-|t|^p}$ and the symmetric $p$-stable measure have this property for $p\in(0,1]$, whereas they do not for $p\in(1,2)$. Notice that the corresponding product measures, apart from $\mu_1^n$, are not log-concave. We note that extending the B-inequality to $\mu_p^n$, where $p>2$, is of importance. For instance, it has been proven by Saroglou \cite{Sar15} that the weak B-inequality for $\mu_\infty^n$ (that is, the uniform measure on the unit cube $B_\infty^n$) would imply the conjectured logarithmic Brunn-Minkowski inequality (see \cite{BLYZ12}) in its full generality. The proof of Theorem \ref{thm:Bineqmixt} will be given in Section \ref{sec:4}.

An application of the B-inequality for Gaussian measure is a small ball probability estimate due to Lata\l a and Oleszkiewicz \cite{LO05}. For a symmetric convex set $K$ denote by $r(K)$ its inradius, i.e. the largest $r>0$ such that $rB_2^n \subseteq K$. In \cite{LO05}, the authors used Theorem \ref{thm:Bineq} along with the Gaussian isoperimetric inequality (see, e.g., \cite[Theorem~3.1.9]{AGM15}) to prove that if $K\subseteq\R^n$ is a symmetric convex set with $\gamma_n(K) \leq 1/2$, then
\begin{equation} \label{eq:smallball1}
\gamma_n(tK) \leq (2t)^{\frac{r(K)^2}{4}} \gamma_n(K), \ \ \ \mbox{for every } t\in[0,1].
\end{equation}
Using Theorem \ref{thm:Bineqmixt} and an isoperimetric-type estimate of Bobkov and Houdr\'e from \cite{BH97} we deduce the following corollary.

\begin{corollary} \label{cor:smallball}
Let $K$ be a symmetric convex set in $\R^n$ such that $\mu_1^n(K)\leq 1/2$. Then
\begin{equation} \label{eq:smallball2}
\mu_1^n(tK) \leq t^{\frac{r(K)}{2\sqrt{6}}} \mu_1^n(K), \ \ \ \mbox{for every } t\in[0,1].
\end{equation}
\end{corollary}

Our next result is an extension of the Gaussian correlation inequality, which was recently proven by Royen in \cite{Roy14} (see also \cite{LM15} for a very clear exposition of Royen's proof and the references therein for the history of the problem).

\begin{theorem} [Gaussian correlation inequality, \cite{Roy14}] \label{thm:Gausscorr}
For any centered Gaussian measure $\gamma$ on $\R^n$ and symmetric convex sets $K,L$ in $\R^n$ we have
\begin{equation} \label{eq:correlation}
\gamma(K\cap L) \geq \gamma(K) \gamma(L).
\end{equation}
\end{theorem}

This inequality admits a straightforward extension to products of laws of Gaussian mixtures.

\begin{theorem} \label{thm:correlation}
Let $X_1,\ldots,X_n$ be Gaussian mixtures and denote by $\mu_i$ the law of $X_i$. Then, for $\mu=\mu_1\times\cdots\times\mu_n$ and any symmetric convex sets $K,L$ in $\R^n$ we have
\begin{equation} \label{eq:correlationmixt}
\mu(K\cap L) \geq \mu(K) \mu(L).
\end{equation}
\end{theorem}

This theorem implies that the correlation inequality \eqref{eq:correlationmixt} holds for the product measure $\mu_p^n$ as well as for all symmetric $p$-stable laws on $\R^n$, where $p\in(0,2)$ (Corollary \ref{cor:corrstable}). In particular, the multivariate Cauchy distribution, which is a rotationally invariant $1$-stable distribution on $\R^n$ defined as $\diff\mu(x)=c_n(1+\|x\|_2^2)^{-\frac{n+1}{2}}\diff x$, satisfies the inequality \eqref{eq:correlationmixt}. In \cite{Mem15}, Memarian proved partial results in this direction and noticed that such inequalities are equivalent to correlation-type inequalities on the unit sphere $S^{n-1}$. We will recap his argument in Section \ref{sec:5}. Let $S^{n-1}_+\subseteq S^{n-1}$ be the open upper hemisphere, i.e. $S^{n-1}_+ = S^{n-1} \cap \{x\in\R^n: \ x_n>0\}$ whose pole is the point $p=(0,\ldots,0,1)$. A subset $A\subseteq S_+^{n-1}$ is called geodesically convex if for any two points $x,y\in A$ the shortest arc of the great circle joining $x,y$ is contained in $A$. Furthermore, $A$ is called symmetric (with respect to the pole $p$) if for any $x\in A$, the point $x^*\neq x$ which lies on the great circle joining $x$ and $p$ and satisfies $d_{S^{n-1}}(x,p) = d_{S^{n-1}}(p,x^*)$, also belongs in $A$. Here $d_{S^{n-1}}$ denotes the geodesic distance on the sphere.

\begin{corollary} \label{cor:corrsphere}
Let $S_+^{n-1}\subseteq S^{n-1}$ be the open upper hemisphere. Then for every symmetric geodesically convex sets $K,L$ in $S_+^{n-1}$ we have
\begin{equation} \label{eq:corrsphere}
|K\cap L| \cdot |S_+^{n-1}| \geq |K|\cdot |L|,
\end{equation}
where $|\cdot|$ denotes the surface area measure on $S^{n-1}$.
\end{corollary}

Finally, we want to stress that one cannot expect that all geometric properties of the Gaussian measure will extend mutatis mutandis to Gaussian mixtures. For example, it has been proven by Bobkov and Houdr\'e in \cite{BH96} that the Gaussian isoperimetric inequality actually characterizes Gaussian measures. Nevertheless, it might be the case that there are many more that admit such an extension.


\subsection{Sections and projections of $B_q^n$} The study of quantitative parameters of sections and projections of convex bodies is a classical topic in convex geometry (for example, see the monograph \cite{Kol05}). As a first application, we revisit two well known theorems and reprove them using some relevant Gaussian mixture representations.

Denote by $H_1$ the hyperplane $(1,0,\ldots,0)^\perp$ and by $H_n$ the hyperplane $(1,\ldots,1)^\perp$. It has been proven by Barthe and Naor in \cite{BN02} that for any $q\in(2,\infty]$ and any hyperplane $H\subseteq\R^n$ we have
\begin{equation} \label{eq:barthenaor}
|\proj_{H_1} B_q^n| \leq |\proj_H B_q^n| \leq |\proj_{H_n} B_q^n|,
\end{equation}
where $|\cdot|$ denotes Lebesgue measure. To deduce this, they proved that for any $q\in[1,\infty]$, if $X_1,\ldots,X_n$ are i.i.d. random variables with density 
\begin{equation}
f_q(t) = c_q |t|^{\frac{2-q}{q-1}} e^{-|t|^{\frac{q}{q-1}}}, \ \ \ t\in\R,
\end{equation}
then the volume of hyperplane projections of $B_q^n$ can be expressed as
\begin{equation} \label{eq:reprbn}
|\proj_{a^\perp} B_q^n| = \alpha_{q,n} \mb{E} \Big| \sum_{i=1}^n a_i X_i \Big|,
\end{equation}
where $a=(a_1,\ldots,a_n)$ is a unit vector and $\alpha_{q,n}$ is a positive constant. It immediately follows from the characterization given in Theorem \ref{thm:charact} that for $q\geq2$ the random variables $X_i$ are Gaussian mixtures and thus, from Theorem \ref{thm:khin} (with $p=1$), we deduce the following strengthening of \eqref{eq:barthenaor}.
\begin{corollary} \label{cor:projections}
Fix $q\in(2,\infty]$. For two unit vectors $a=(a_1,\ldots,a_n), b=(b_1,\ldots,b_n)$ in $\R^n$ we have
\begin{equation} \label{eq:projections}
(a_1^2,\ldots,a_n^2) \preceq (b_1^2,\ldots,b_n^2) \ \Longrightarrow \ |\proj_{a^\perp}B_q^n| \geq |\proj_{b^\perp} B_q^n|.
\end{equation}
\end{corollary}

We now turn to the dual question for sections. Meyer and Pajor and later Koldobsky (see \cite{MP88}, \cite{Kol98}) proved that for any $q\in(0,2)$ and any hyperplane $H\subseteq\R^n$
\begin{equation} \label{eq:koldobsky}
|B_q^n \cap H_n| \leq |B_q^n \cap H| \leq |B_q^n \cap H_1|.
\end{equation}
More precisely, in \cite{MP88} the authors proved the upper bound of \eqref{eq:koldobsky} for $q\in[1,2)$ and the lower bound for $q=1$ and posed a conjecture that would imply \eqref{eq:koldobsky} for any $q\in(0,2)$; this was later confirmed in \cite{Kol98}. The main ingredients in Koldobsky's proof of \eqref{eq:koldobsky} were a general representation of the volume of hyperplane sections of a convex body in terms of the Fourier transform of the underlying norm and an elegant lemma about symmetric $q$-stable densities. Using a different approach, we prove the analogue of Corollary \ref{cor:projections} for sections.
\begin{corollary} \label{cor:sections}
Fix $q\in(0,2)$. For two unit vectors $a=(a_1,\ldots,a_n), b=(b_1,\ldots,b_n)$ in $\R^n$ we have
\begin{equation} \label{eq:sections}
(a_1^2,\ldots,a_n^2)\preceq (b_1^2,\ldots,b_n^2) \ \Longrightarrow \ |B_q^n \cap a^\perp| \leq |B_q^n\cap b^\perp|.
\end{equation}
\end{corollary}

In fact, Corollary \ref{cor:sections} will follow from a more general comparison of Gaussian parameters of sections which is in the spirit of \cite{BGMN05}. For a hyperplane $H\subseteq\R^n$ and a convex body $K\subseteq\R^n$ denote by $\|\cdot\|_{K\cap H}$ the norm on $H$ associated with the convex body $K\cap H$.

\begin{theorem} \label{thm:laplacegaussian}
Fix $q\in(0,2)$. For a unit vector $\theta\in\R^n$ let $G_\theta$ be a standard Gaussian random vector on the hyperplane $\theta^\perp$. Then for every $\lambda>0$ and unit vectors $a=(a_1,\ldots,a_n), b=(b_1,\ldots,b_n)$ in $\R^n$ we have
\begin{equation} \label{eq:laplacegaussiancomp}
(a_1^2,\ldots,a_n^2) \preceq (b_1^2,\ldots,b_n^2) \ \Longrightarrow \ \mb{E} e^{-\lambda\|G_a\|_{B_q^n\cap a^\perp}^q} \leq \mb{E} e^{-\lambda \|G_b\|_{B_q^n\cap b^\perp}^q}.
\end{equation}
\end{theorem}

In \cite{BGMN05}, the authors used a different method to prove that for any $q\in(0,2)$ and $\lambda>0$ the Gaussian parameters appearing in \eqref{eq:laplacegaussiancomp} are maximized when $a=e_1$. As explained there, such inequalities imply the comparison of various other parameters of sections and projections of $B_q^n$, most notably the volume (Corollary \ref{cor:sections}) and the mean width. Recall that for a symmetric convex body $K$ in $\R^n$ the support function $h_K:S^{n-1}\to\R_+$ is defined as $h_K(\theta) = \max_{x\in K}\langle x,\theta\rangle$ and the mean width is
$$w(K) = \int_{S^{n-1}} h_K(\theta) \diff\sigma(\theta),$$
where $\sigma$ is the rotationally invariant probability measure on the unit sphere $S^{n-1}$. Exploiting the duality between sections and projections we deduce the following corollary.

\begin{corollary} \label{cor:width}
Fix $q\in(2,\infty]$ and let $H\subseteq\R^n$ be a hyperplane. Then
\begin{equation} \label{eq:width}
w(\proj_{H_1} B_q^n) \leq w(\proj_H B_q^n) \leq w(\proj_{H_n} B_q^n).
\end{equation}
\end{corollary}

The lower bound in \eqref{eq:width} was first obtained in \cite{BGMN05}, where the authors also proved that for any $q\in(0,2)$ and any hyperplane $H\subseteq\R^n$
\begin{equation} \label{eq:widthpreviousbounds}
w(\proj_H B_q^n)\leq w(\proj_{H_1}B_q^n).
\end{equation}
Given this result and Corollary \ref{cor:width}, what remains to be understood is which hyperplane projections of $B_q^n$ have minimal mean width for $q\in(0,2)$, similarly to the study of volume.  We will provide the proof of Theorem \ref{thm:laplacegaussian} and its consequences in Section \ref{sec:6}.
\end{section}


\begin{section}{Proof of Theorem \ref{thm:charact} and examples} \label{sec:2}

Here we establish some initial facts about Gaussian mixtures, prove the characterization presented in the introduction and use it to provide relevant examples. 

Let $X$ be a Gaussian mixture with the same distribution as $YZ$, where $Y$ is positive and $Z$ is an independent standard Gaussian random variable; denote by $\nu$ the law of $Y$. Clearly $X$ is symmetric. Furthermore, for a Borel set $A\subseteq\R$ we have
\begin{equation}
\mb{P}(X\in A) = \mb{P}(YZ\in A) = \int_0^\infty \mb{P}(yZ\in A) \diff\nu(y) = \int_A \int_0^\infty \frac{1}{\sqrt{2\pi}y} e^{-\frac{x^2}{2y^2}}\diff\nu(y)\diff x,
\end{equation}
which immediately implies that $X$ has a density
\begin{equation} \label{eq:density}
f(x) = \frac{1}{\sqrt{2\pi}} \int_0^\infty e^{-\frac{x^2}{2y^2}} \frac{\diff\nu(y)}{y}.
\end{equation}
We now proceed with the proof of Theorem \ref{thm:charact}.

\medskip

\noindent {\it Proof of Theorem \ref{thm:charact}.} Let $X$ be a symmetric random variable with density $f$ such that the function $x\mapsto f(\sqrt{x})$ is completely monotonic. By Bernstein's theorem, there exists a non-negative Borel measure $\mu$ supported on $[0,\infty)$ such that
\begin{equation} \label{eq:laplacetransform}
f(\sqrt{x}) = \int_0^\infty e^{-tx} \diff\mu(t), \ \ \ \mbox{for every } x>0
\end{equation}
or, equivalently, $f(x) = \int_0^\infty e^{-tx^2} \diff\mu(t)$ for every $x\in\R$.
Notice that $\mu(\{0\})=0$, because otherwise $f$ would not be integrable. Now, for a subset $A\subseteq\R$ we have
\begin{equation} \label{eq:mixture}
\begin{split}
\mb{P}(X\in A) & = \int_A \int_0^\infty e^{-tx^2} \diff\mu(t) \diff x = \int_0^\infty \int_A e^{-tx^2} \diff x \diff\mu(t) \\ & = \int_0^\infty \int_{\sqrt{2t}A} \frac{1}{\sqrt{2\pi}} e^{-x^2/2} \diff x \ \sqrt{\frac{\pi}{t}}\diff\mu(t) = \int_0^\infty \gamma_n(\sqrt{2t}A) \diff\nu(t),
\end{split}
\end{equation}
where $\diff\nu(t) = \sqrt{\frac{\pi}{t}}\diff\mu(t)$. In particular, choosing $A=\R$, we deduce that $\nu$ is a probability measure, supported on $(0,\infty)$. Let $V$ be a random variable distributed according to $\nu$; clearly $V$ is positive almost surely. Define $Y= \frac{1}{\sqrt{2V}}$ and let $Z$ be a standard Gaussian random variable, independent of $Y$. Then \eqref{eq:mixture} implies that
\begin{equation*}
\begin{split}
\mb{P}(YZ \in A) & = \mb{P}\left( \frac{1}{\sqrt{2V}} \cdot Z \in A\right) = \int_0^\infty \gamma_n(\sqrt{2t}A) \diff\nu(t) = \mb{P}(X\in A),
\end{split}
\end{equation*}
that is, $X$ has the same distribution as the product $YZ$. The converse implication readily follows from \eqref{eq:density} and Bernstein's theorem after a change of variables. \hfill$\Box$

\medskip

Before applying Theorem \ref{thm:charact} we first provide some examples of completely monotonic functions. Direct differentiation shows that the functions $e^{-\alpha x}, x^{-\alpha}$ and $(1+x)^{-\alpha}$, where $\alpha>0$, are completely monotonic on $(0,\infty)$ and a straightforward induction proves that the same holds for $e^{-x^\beta}$, where $\beta\in(0,1]$. The same argument implies that if $g$ is a completely monotonic function on $(0,\infty)$ and $h$ is positive and has a completely monotonic derivative on $(0,\infty)$, then $g\circ h$ is also completely monotonic on $(0,\infty)$. Moreover, one can easily see that products of completely monotonic functions themselves are completely monotonic.

Combining the last example with Theorem \ref{thm:charact}, we get that for every $p\in(0,2]$ the random variable with density proportional to $e^{-|t|^p}$ is a Gaussian mixture. Recall that we denote by $\mu_p$ the probability measure with density $c_p e^{-|t|^p}$, $p>0$, where $c_p = (2\Gamma(1+1/p))^{-1}$, and $\mu_p^n = \mu_p^{\otimes n}$.  Student’s $t$ random variables provide another example of Gaussian mixtures. Furthermore, it is a classical fact that symmetric $p$-stable random variables, where $p\in(0,2]$, are Gaussian mixtures. For these measures we can describe the positive factor in their Gaussian mixture representation. Recall that a positive random variable $W$ with Laplace transform $\mb{E} e^{-tW} = e^{-c t^\alpha}$, where $\alpha\in(0,1)$ and $c>0$, is called a positive $\alpha$-stable random variable. Standard positive $\alpha$-stable random variables correspond to $c=1$; we denote their density by $g_\alpha$. 

\begin{lemma} \label{lemma:computations}
Fix $p\in(0,2)$ and let $Z$ be a standard Gaussian random variable.
\begin{enumerate} [(i)]
\item If $V_{p/2}$ has density proportional to $t^{-1/2} g_{p/2}(t)$ and is independent of $Z$, then $(2V_{p/2})^{-1/2}Z$ has density $c_pe^{-|t|^p}$.
\item If $W_{p/2}$ is a standard positive $p/2$-stable random variable and is independent of $Z$, then $(2W_{p/2})^{1/2} Z$ is a standard symmetric $p$-stable random variable.
\end{enumerate}
\end{lemma}

\begin{proof} To show (i), we shall decompose a symmetric random variable with density $c_p e^{-|x|^p}$ into a product of two independent random variables: a positive one and a standard Gaussian. To this end, denote by $\mu$ the measure in the representation \eqref{eq:laplacetransform} written for the density $c_p e^{-|x|^p}$, that is
$$c_p e^{-x^{p/2}} = \int_0^\infty e^{-tx}\diff\mu(t), \ \ \ x>0.$$
Therefore, the Laplace transform of $c_p^{-1}\mu$ is $e^{-x^{p/2}}$, which implies that $c_p^{-1}\mu$ is a standard positive $p/2$-stable measure with density $g_{p/2}$. Now, an inspection of the proof of Theorem \ref{thm:charact}, reveals that the positive factor $Y$ in the Gaussian mixture representation is $Y=(2V)^{-1/2}$, where $V$ has law $\sqrt{\frac{\pi}{t}}\diff\mu(t)$, so in this case the density of $V$ is indeed proportional to $t^{-1/2} g_{p/2}(t)$, as required.

On the other hand, (ii) is a straightforward characteristic function computation. Using the independence of $W_{p/2}$ and $Z$ we get 
\begin{equation*}
\begin{split}
\mb{E} e^{i\sqrt{2}tW_{p/2}^{1/2} Z} = \mb{E}_{W_{p/2}} \mb{E}_Z e^{i\sqrt{2}tW_{p/2}^{1/2}Z} = \mb{E} e^{-t^2W_{p/2}} = e^{-t^p},
\end{split}
\end{equation*}
which concludes the proof of the lemma.
\end{proof}

Lemma \ref{lemma:computations} will be useful in Section \ref{sec:4}. It can be checked by doing further computations that when $p=1$ these Gaussian mixture representations have the following explicit forms.
\begin{enumerate} [(i)]
\item Let $\mathcal{E}$ be an exponential random variable (that is, a random variable with density $e^{-t}{\bf1}_{t>0}$) and $Z$ a standard Gaussian random variable, independent of $\mathcal{E}$. Then the product $\sqrt{2\mathcal{E}} Z$ has density $\frac{1}{2}e^{-|t|}$, $t\in\R$ (symmetric exponential density).
\item Let $Z_1,Z_2$ be independent standard Gaussian random variables. Then the quotient $Z_1/|Z_2|$ is distributed according to the Cauchy distribution with density $\frac{1}{\pi(1+x^2)}$, which is the symmetric 1-stable distribution.
\end{enumerate}

\begin{remark}
It was noted in \cite[p.~223]{BN02} that for an infinitely differentiable integrable function $f:(0,\infty)\to\R$, the function $x\mapsto f(\sqrt{x})$ is completely monotonic if and only if $x\mapsto\widehat{f}(\sqrt{x})$ is completely monotonic, where $\widehat{f}$ is the Fourier transform of $f$. Applying this to the density $c_pe^{-|t|^p}$ and then using Theorem \ref{thm:charact} yields that symmetric $p$-stable random variables are Gaussian mixtures, as was also proven above.
\end{remark}

\end{section}


\begin{section}{Moment and entropy comparison} \label{sec:3}
For the proofs of this section, we will use an elementary result of Marshall and Proschan from \cite{MP65} which reads as follows. Let $\phi:\R^n\to\R$ be a convex function, symmetric under permutations of its $n$ arguments. Let $X_1,\ldots,X_n$ be interchangeable random variables, that is, random variables whose joint distribution is invariant under permutations of its coordinates. Then for two vectors $(a_1,\ldots,a_n), (b_1,\ldots,b_n)\in\R^n$ we have
\begin{equation} \label{eq:marpro}
(a_1,\ldots,a_n) \preceq (b_1,\ldots,b_n) \ \Longrightarrow \ \mb{E}\phi(a_1X_1,\ldots,a_nX_n) \leq \mb{E} \phi(b_1X_1,\ldots,b_nX_n)
\end{equation}
or, in other words, the function $\R^n\ni(a_1,\ldots,a_n) \mapsto \mb{E}\phi(a_1X_1,\ldots,a_nX_n)$ is Schur convex. If $\phi$ is concave, then the second inequality in \eqref{eq:marpro} is reversed, i.e. the function above is Schur concave. This result follows directly from the fact that a convex (respectively concave) function that is symmetric under permutations of its arguments is Schur convex (respectively concave), which, in turn, is a consequence of the following simple property. If $a=(a_1,\ldots,a_n), b=(b_1,\ldots,b_n)\in\R^n$ then
$$a \preceq b \ \Longleftrightarrow \ a \in \mathrm{conv}\big\{ (b_{\sigma(1)},\ldots,b_{\sigma(n)}): \ \sigma \mbox{ is a permutation of } \{1,\ldots,n\}\big\}, $$
where $\mathrm{conv}(A)$ denotes the convex hull of a set $A\subseteq\R^n$ (for details, see \cite{MO79}).

We start with the comparison of moments of Gaussian mixtures.

\medskip

\noindent {\it Proof of Theorem \ref{thm:khin}.} Fix $p>-1$, $p\neq0$. Let $X$ be a Gaussian mixture and $X_1,\ldots, X_n$ be independent copies of $X$. Since each $X_i$ is a Gaussian mixture, there exist i.i.d. positive random variables $Y_1,\ldots,Y_n$ and independent standard Gaussian random variables $Z_1,\ldots,Z_n$ such that $X_i$ has the same distribution as the product $Y_i Z_i$. For $a_1,\ldots,a_n\in\R$ the joint independence of the $Y_i, Z_j$ implies that
$$\mb{E}\Big|\sum_{i=1}^n a_iX_i\Big|^p = \mb{E}\Big| \sum_{i=1}^n a_iY_iZ_i\Big|^p = \mb{E} \Big| \Big( \sum_{i=1}^n a_i^2 Y_i^2\Big)^{1/2} Z \Big|^p = \gamma_p^p \cdot \mb{E} \Big| \sum_{i=1}^n a_i^2 Y_i^2\Big|^{p/2},$$
where $Z$ is a standard Gaussian random variable independent of all the $Y_i$ and $\gamma_p =(\mb{E}|Z|^p)^{1/p}$. The conclusion now follows directly from Marshall and Proschan's result \eqref{eq:marpro} since $t\mapsto t^{p/2}$ is convex for $p\in(-1,0)\cup[2,\infty)$ and concave for $p\in(0,2)$. Notice that when the exponent $1/p$ is negative, the resulting norm becomes Schur concave. The result for $p=0$ is proven similarly. \hfill$\Box$

\medskip

The derivation of sharp constants in the corresponding Khintchine inequalities is now straightforward.

\begin{corollary} \label{cor:khinconstants}
Let $X$ be a Gaussian mixture and $X_1,\ldots,X_n$ be independent copies of $X$. Then, for every $p\in(-1,\infty)$ and $a_1,\ldots,a_n$ in $\R$ we have
\begin{equation} \label{eq:mixtkhin2}
A_p \Big\|\sum_{i=1}^n a_iX_i \Big\|_2 \leq \Big\|\sum_{i=1}^n a_i X_i \Big\|_p \leq B_p \Big\|\sum_{i=1}^n a_iX_i \Big\|_2,
\end{equation}
where
\begin{equation} \label{eq:mixtconstants}
A_p = \begin{cases} \frac{\|X\|_p}{\|X\|_2},& p\in(-1,2) \\ \gamma_p,& p\in [2,\infty) \end{cases} \ \ \ \mbox{and} \ \ \ B_p = \begin{cases} \gamma_p,& p\in(-1,2) \\ \frac{\|X\|_p}{\|X\|_2},& p\in [2,\infty) \end{cases},
\end{equation}
provided that all the moments exist. Here $\gamma_p = \sqrt{2} \left( \frac{\Gamma\left(\frac{p+1}{2}\right)}{\sqrt{\pi}}\right)^{1/p}$ is the $p$-th moment of a standard Gaussian random variable. These constants are sharp.
\end{corollary}

\begin{proof} 
We can clearly assume that $(a_1,\ldots,a_n)$ is a unit vector. We will prove the statement for $p\geq2$; the case $p\in(-1,2)$ is identical. The Schur convexity statement of Theorem \ref{thm:khin} along with \eqref{eq:order} implies that
\begin{equation} \label{eq:monotonicity}
\Big\|\frac{X_1+\cdots+X_n}{\sqrt{n}} \Big\|_p \leq \Big\| \sum_{i=1}^n a_i X_i\Big\|_p \leq \|X_1\|_p.
\end{equation}
Applying this for $a_1=\cdots=a_{n-1}=(n-1)^{-1/2}$ and $a_n=0$, where $n\geq2$, shows that the quantity on the left-hand side is decreasing in $n$ and the central limit theorem implies that
$$\gamma_p \|X\|_2 \leq \Big\| \sum_{i=1}^n a_iX_i\Big\|_p \leq \|X\|_p,$$
which is equivalent to
$$\gamma_p \Big\| \sum_{i=1}^n a_iX_i\Big\|_2 \leq \Big\| \sum_{i=1}^n a_iX_i\Big\|_p \leq \frac{\|X\|_p}{\|X\|_2}  \Big\| \sum_{i=1}^n a_iX_i\Big\|_2.$$
The sharpness of the constants is evident.
\end{proof}

For the proof of Corollary \ref{cor:ballkhin} we need to exploit two results about the geometry of $B_q^n$ which are probabilistic in nature. Let $Y_1,\ldots, Y_n$ be i.i.d. random variables distributed according to $\mu_q$ and write $Y=(Y_1,\ldots,Y_n)$. 

We denote by $S$ the random variable $\big(\sum_{i=1}^n |Y_i|^q\big)^{1/q}$. As explained in the introduction, the main ingredient of the proof of Corollary \ref{cor:ballkhin} is a representation for the uniform measure on $B_q^n$ discovered in \cite{BGMN05} that reads as follows. Let $\mathcal{E}$ be an exponential random variable (that is, the density of $\mathcal{E}$ is $e^{-t}{\bf 1}_{t>0}$) independent of the $Y_i$. Then the random vector 
$$\Big(\frac{Y_1}{(S^q+\mathcal{E})^{1/q}},\ldots,\frac{Y_n}{(S^q+\mathcal{E})^{1/q}}\Big)$$ 
is uniformly distributed on $B_q^n$. Furthermore, we will need a result of Schechtman and Zinn from \cite{SZ90}, also independently proven by Rachev and R\"uschendorf in \cite{RR91}, which asserts that the random variables $S$ and $\frac{Y}{S}$ are independent.

\medskip

\noindent {\it Proof of Corollary \ref{cor:ballkhin}.} Recall that $X=(X_1,\ldots,X_n)$ is a random vector uniformly distributed on $B_q^n$ and let $Y_1,\ldots,Y_n, S$ and $\mathcal{E}$ be as above. For the reader's convenience we repeat the following computation from \cite{BGMN05}. Using the representation described before and the independence of $S$ and $\frac{Y}{S}$ we get
\begin{equation*}
\mb{E} \Big| \sum_{i=1}^n a_iX_i \Big|^p = \mb{E} \Big| \frac{1}{(S^q+\mathcal{E})^{1/q}} \sum_{i=1}^n a_iY_i \Big|^p = \mb{E} \Big| \frac{S}{(S^q+\mathcal{E})^{1/q}}\Big|^p \mb{E} \Big|\sum_{i=1}^n a_i \frac{Y_i}{S}\Big|^p.
\end{equation*}
Then, again by independence, $\mb{E}\big| \sum_{i=1}^n a_i \frac{Y_i}{S}\big|^p \mb{E}|S|^p = \mb{E} \big| \sum_{i=1}^n a_iY_i\big|^p$ and thus
\begin{equation} \label{eq:proportional}
\mb{E}\Big| \sum_{i=1}^n a_i X_i \Big|^p = \frac{1}{\mb{E}|S|^p} \mb{E}\Big| \frac{S}{(S^q+\mathcal{E})^{1/q}}\Big|^p \mb{E} \Big| \sum_{i=1}^n a_iY_i\Big|^p = c(p,q,n) \mb{E}\Big|\sum_{i=1}^n a_iY_i\Big|^p,
\end{equation}
where $c(p,q,n)>0$ is independent of the vector $(a_1,\ldots,a_n)$. In other words, the moments of linear functionals applied to the vector $X$ are proportional to the moments of the same linear functionals applied to $Y$. In view of Theorem \ref{thm:khin} and of the fact that $Y_1,\ldots,Y_n$ are i.i.d. Gaussian mixtures, this property readily implies Corollary \ref{cor:ballkhin}.
\hfill$\Box$

\medskip

Similarly to Corollary \ref{cor:khinconstants}, it is straightforward to deduce the sharp constants for Khintchine inequalities on $B_q^n$.

\begin{corollary} \label{cor:ballconstants} \footnote{This corollary in the published version of the paper contains a small error that has been corrected here.}
Fix $q\in (0,2]$ and let $X=(X_1,\ldots,X_n)$ be a random vector, uniformly distributed on $B_q^n$. Then, for every $p\in(-1,\infty)$ and $a_1,\ldots,a_n$ in $\R$ we have
\begin{equation} \label{eq:ballkhin2}
A_{p,q,n} \Big\| \sum_{i=1}^n a_iX_i\Big\|_2 \leq \Big\| \sum_{i=1}^n a_iX_i\Big\|_p \leq B_{p,q,n} \Big\|\sum_{i=1}^n a_iX_i\Big\|_2,
\end{equation}
where
\begin{equation} \label{eq:ballconstants}
A_{p,q,n} = \begin{cases} \frac{\|X_1\|_p}{\|X_1\|_2},& p\in(-1,2) \\ \\ \frac{\|X_1\|_p}{\|X_1\|_2} \frac{\kappa_2}{\kappa_p} \gamma_p,& p\in [2,\infty) \end{cases} \ \ \ \mbox{and} \ \ \ B_{p,q,n} = \begin{cases} \frac{\|X_1\|_p}{\|X_1\|_2} \frac{\kappa_2}{\kappa_p} \gamma_p,& p\in(-1,2) \\ \\ \frac{\|X_1\|_p}{\|X_1\|_2},& p\in [2,\infty) \end{cases}
\end{equation}
and for $r>-1$
\begin{equation} \label{eq:ballnorms}
\|X_1\|_r =  \left(\frac{B\Big( \frac{r+1}{q},\frac{n+q-1}{q}\Big)}{B\Big( \frac{1}{q},\frac{n+q-1}{q}\Big)}\right)^{1/r} \ \ \mbox{and} \ \ \kappa_r = \left( \frac{\Gamma\Big( \frac{r+1}{q}\Big)}{\Gamma\Big(\frac{1}{q}\Big)}\right)^{1/r}.
\end{equation}
This value of $A_{p,q,n}$ is sharp for $p \in (-1,2)$ and of $B_{p,q,n}$ for $p \in [2,\infty)$.
\end{corollary}

\begin{proof}
Let $Y_1,\ldots,Y_n$ be as in the proof of Corollary \ref{cor:ballkhin}. By \eqref{eq:proportional}, for every $a_1,\ldots,a_n\in\R$, we have
\begin{equation} \label{eq:proportional times two}
\Big\|\sum_{i=1}^n a_i X_i\Big\|_p = c(p,q,n)^{1/p} \Big\|\sum_{i=1}^n a_i Y_i\Big\|_p \ \ \mbox{and} \ \ \Big\|\sum_{i=1}^n a_i X_i\Big\|_2 = c(2,q,n)^{1/2} \Big\|\sum_{i=1}^n a_i Y_i\Big\|_2.
\end{equation} 
Notice that for every $r\in(-1,\infty)$,
\begin{equation}
\|Y_1\|_r^r = \frac{1}{\Gamma\big(1+\frac{1}{q}\big)} \int_0^\infty t^r e^{-t^q}\diff t = \frac{1}{\Gamma\big(\frac{1}{q}\big)} \int_0^\infty t^{\frac{r+1}{q}-1} e^{-t}\diff t = \frac{\Gamma\big( \frac{r+1}{q}\big)}{\Gamma\big(\frac{1}{q}\big)} = \kappa_r^r.
\end{equation}
Therefore, applying \eqref{eq:proportional} to $a_1=1$ and $a_2=\cdots=a_n=0$, we conclude that
\begin{equation} \label{eq:compute crq}
c(r,q,n)^{1/r} = \frac{\|X_1\|_r}{\|Y_1\|_r} = \frac{\|X_1\|_r}{\kappa_r}.
\end{equation}
Combining \eqref{eq:proportional times two}, \eqref{eq:compute crq} and Corollary \ref{cor:ballkhin}, we derive \eqref{eq:ballconstants}. To deduce \eqref{eq:ballnorms}, notice that $X_1$ has density $f(x) = c_{q,n} (1-|x|^q)^{\frac{n-1}{q}} {\bf 1}_{|x|\leq1}$ where
$$c_{q,n}^{-1} = 2 \int_0^1 (1-x^q)^{\frac{n-1}{q}}\diff x = \frac{2}{q} B\Big(\frac{1}{q},\frac{n+q-1}{q} \Big).$$
Thus, for every $r>0$,
$$\|X_1\|_r = \Big(2c_{q,n} \int_0^1 x^r (1-x^q)^{\frac{n-1}{q}} \diff x \Big)^{1/r} =\left(\frac{B\Big( \frac{r+1}{q},\frac{n+q-1}{q}\Big)}{B\Big( \frac{1}{q},\frac{n+q-1}{q}\Big)}\right)^{1/r},$$
which completes the proof.
\end{proof}

\noindent {\it Remark.}
Fix $q \in (0,2]$. It can be checked that for $p \in [2,\infty)$, the sequence $(\frac{\|X_1\|_p}{\|X_1\|_2})_{n\geq 1}$ is increasing, whereas for $p \in (-1,2)$, it is decreasing. Consequently, for $p \in [2,\infty)$, $\inf_{n \geq 1} A_{p,q,n} = A_{p,q,1} = \frac{3^{1/2}}{(p+1)^{1/p}}\frac{\kappa_2}{\kappa_p}\gamma_p$ and for $p \in (-1,2)$, $\sup_{n \geq 1} B_{p,q,n} = B_{p,q,1} = \frac{3^{1/2}}{(p+1)^{1/p}}\frac{\kappa_2}{\kappa_p}\gamma_p$.

\medskip

We now turn to comparison of entropy.

\medskip

\noindent {\it Proof of Theorem \ref{thm:entropycomp}.} Let $X$ be a Gaussian mixture and $X_1,\ldots,X_n$ independent copies of $X$. There exist i.i.d. positive random variables $Y_1,\ldots,Y_n$ and independent standard Gaussian random variables $Z_1,\ldots,Z_n$ such that $X_i$ has the same distribution as the product $Y_iZ_i$. For a vector $\theta=(\theta_1,\ldots,\theta_n)\in\R^n$ denote by $X_\theta$ the random variable $\sum_{i=1}^n \theta_i X_i$ and by $f_\theta$ the density of $X_\theta$. Since $X_\theta$ is itself a Gaussian mixture, Theorem \ref{thm:charact} implies that the function $x\mapsto f_\theta(\sqrt{x})$ is completely monotonic. Consequently, there exists a measure $\mu_\theta$ on $[0,\infty)$ so that
$$f_\theta(\sqrt{x}) = \int_0^\infty e^{-tx} \diff\mu_\theta(t), \ \ \ \mbox{for every } x>0.$$
It now immediately follows from H\"older's inequality that for $x,y>0$ and $\lambda\in(0,1)$ we have
\begin{equation*}
\begin{split}
f_\theta(\sqrt{\lambda x+(1-\lambda) y}) & = \int_0^\infty (e^{-tx})^\lambda (e^{-ty})^{1-\lambda} \diff\mu_\theta(t) \\ & \leq \Big( \int_0^\infty e^{-tx}\diff\mu_\theta(t)\Big)^\lambda \Big( \int_0^\infty e^{-ty}\diff\mu_\theta(t) \Big)^{1-\lambda} = f_\theta(\sqrt{x})^\lambda f_\theta(\sqrt{y})^{1-\lambda}
\end{split}
\end{equation*}
or, in other words, the function $\varphi_\theta(x) = -\log f_\theta(\sqrt{x})$ is concave.

Let $a=(a_1,\ldots,a_n), b=(b_1,\ldots,b_n)\in\R^n$ be such that $(a_1^2,\ldots,a_n^2)\preceq(b_1^2,\ldots,b_n^2)$. We first consider the case of Shannon entropy, i.e. $\alpha=1$. Jensen's inequality implies the following well known variational formula
\begin{equation} \label{eq:varentropy}
\ent(X_b)  = \mb{E}[ -\log f_b(X_b) ] = \min \Big\{\mb{E} [-\log g(X_b)] : \ g:\R\to\R_+ \mbox{ is a probability density}\Big\}.
\end{equation}
Thus, using \eqref{eq:varentropy} for $g=f_a$ we get
\begin{equation} \label{eq:usemixture}
\begin{split}
\ent(X_b) & \leq \mb{E}[-\log f_a (X_b)] = \mb{E} \Big[ -\log f_a\Big(\sum_{i=1}^n b_i Y_i Z_i \Big) \Big] \\ & = \mb{E} \Big[ -\log f_a \Big( \Big(\sum_{i=1}^n b_i^2 Y_i^2\Big)^{1/2} Z\Big) \Big] = \mb{E}_Z \mb{E}_{Y} \varphi_a \Big( \sum_{i=1}^n b_i^2 Y_i^2 Z^2\Big),
\end{split}
\end{equation}
where in the last equality we used the fact that $Z$ is independent of the $Y_i$. Now, since $(a_1^2,\ldots,a_n^2)$ is majorized by $(b_1^2,\ldots,b_n^2)$, the concavity of $\varphi_a$ along with Marshall and Proschan's result \eqref{eq:marpro} imply that
$$\mb{E}_{Y} \varphi_a \Big( \sum_{i=1}^n b_i^2 Y_i^2 Z^2 \Big) \leq \mb{E}_{Y} \varphi_a \Big( \sum_{i=1}^n a_i^2 Y_i^2 Z^2 \Big)$$
which, after averaging over $Z$, gives
$$\ent(X_b) \leq \mb{E}\varphi_a \Big( \sum_{i=1}^n a_i^2 Y_i^2 Z^2 \Big) = \mb{E} [-\log f_a(X_a)] = \ent(X_a).$$
For the R\'enyi entropy of order $\alpha$, where $\alpha>1$, we need to prove that
\begin{equation} \label{eq:renyigoal}
\int_\R f_a^\alpha(x)\diff x \leq \int_\R f_b^\alpha(x) \diff x.
\end{equation}
Notice that, as before, we can write
\begin{equation} \label{eq:renyiusemixture}
\int_\R f_a^\alpha(x)\diff x = \mb{E} f_a^{\alpha-1}(X_a) = \mb{E}_Z \mb{E}_Y f_a^{\alpha-1} \Big( \Big(\sum_{i=1}^n a_i^2Y_i^2\Big)^{1/2} Z \Big).
\end{equation}
The concavity of $\varphi_a$ implies that, since $\alpha>1$, the function $x\mapsto f_a^{\alpha-1}(\sqrt{x}) = e^{(1-\alpha)\varphi_a(x)}$ is convex and thus from \eqref{eq:marpro} we get
$$\mb{E}_Yf_a^{\alpha-1} \Big( \Big(\sum_{i=1}^n a_i^2Y_i^2\Big)^{1/2} Z \Big)  \leq \mb{E}_Y f_a^{\alpha-1} \Big( \Big(\sum_{i=1}^n b_i^2Y_i^2\Big)^{1/2} Z \Big)$$
which, after integrating with respect to $Z$, gives
\begin{equation} \label{eq:renyimarpro}
\int_\R f_a^\alpha(x) \diff x \leq \mb{E} f_a^{\alpha-1} \Big( \Big( \sum_{i=1}^n b_i^2 Y_i^2\Big)^{1/2} Z\Big) = \mb{E} f_a^{\alpha-1} (X_b) = \int_\R f_a^{\alpha-1}(x) f_b(x)\diff x.
\end{equation}
Finally, H\"older's inequality yields
\begin{equation} \label{eq:renyiholder}
\int_\R f_a^{\alpha-1}(x) f_b(x) \diff x \leq \Big( \int_\R f_a^\alpha(x) \diff x \Big)^{\frac{\alpha-1}{\alpha}} \Big(\int_\R f_b^\alpha(x)\diff x \Big)^{\frac{1}{\alpha}}.
\end{equation}
Combining \eqref{eq:renyimarpro} and \eqref{eq:renyiholder} readily implies \eqref{eq:renyigoal}, i.e. the comparison $h_\alpha(X_a)\geq h_\alpha(X_b)$.
\hfill$\Box$

\begin{remark}
We note that a result of similar nature was proven in the work \cite{Yu08} of Yu, who showed that for every i.i.d. symmetric log-concave random variables $X_1,\ldots,X_n$ the function $(a_1,\ldots,a_n) \mapsto \ent\big(\sum_{i=1}^n a_iX_i\big)$ is Schur convex on $\R^n$. In particular, for every vector $(a_1,\ldots,a_n)\in\R^n$ such that $\sum_{i=1}^n |a_i|=1$ we have
\begin{equation}
\ent\Big(\frac{1}{n}\sum_{i=1}^n X_i\Big) \leq \ent\Big(\sum_{i=1}^n a_iX_i\Big) \leq \ent(X_1).
\end{equation}
The main actors in Yu's argument are the same: the variational principle for entropy \eqref{eq:varentropy} and Marshall and Proschan's comparison result \eqref{eq:marpro} (the log-concavity assumption is paired up with the linear constraint on the coefficients).
\end{remark}

Finally, we proceed with the proof of Proposition \ref{prop:swap}.

\medskip

\noindent {\it Proof of Proposition \ref{prop:swap}.} Let $X_1, X_2$ be independent Gaussian mixtures such that $X_i$ has the same distribution as the product $Y_iZ_i$, for some independent positive random variables $Y_1, Y_2$ and independent standard Gaussian random variables $Z_1, Z_2$. Let $G$ be a centered Gaussian random variable independent of $X_1$ with the same variance as $X_2$. Notice that $X_1+X_2$ has the same distribution as $(Y_1^2+Y_2^2)^{1/2}Z$, whereas $X_1+G$ has the same distribution as $(Y_1^2+\mb{E}Y_2^2)^{1/2} Z$, where $Z$ is a standard Gaussian random variable independent of the $Y_i$. Denote by $f$ the density of $X_1+X_2$ and by $g$ the density of $X_1+G$. Using the variational formula for entropy \eqref{eq:varentropy} we get
\begin{equation*}
\begin{split}
\ent(X_1+X_2) & = \mb{E} [-\log f(X_1+X_2)] \\ & \leq \mb{E}[-\log g(X_1+X_2)] = \mb{E}_{(Y_1,Z)}\mb{E}_{Y_2}[-\log g((Y_1^2+Y_2^2)^{1/2}Z)].
\end{split}
\end{equation*}
Since $X_1+G$ is also a Gaussian mixture, as remarked in the proof of Theorem \ref{thm:entropycomp}, the function $x\mapsto -\log g(\sqrt{x})$ is concave and thus
$$\mb{E}_{Y_2} [-\log g((Y_1^2+Y_2^2)^{1/2}Z)] \leq -\log g((Y_1^2+\mb{E}Y_2^2)^{1/2}Z).$$
Combining the above we deduce that
$$\ent(X_1+X_2) \leq \mb{E} [-\log g((Y_1^2+\mb{E}Y_2^2)^{1/2}Z)] = \mb{E}[-\log g(X_1+G)] = \ent(X_1+G),$$
which concludes the proof.
\hfill$\Box$

\begin{remark}
In light of Proposition \ref{prop:swap}, it could seem that the assumption that $X_1,X_2$ are identically distributed in Question \ref{quest:swap} is redundant. However, this is not the case. Let $X_1,X_2$ be independent symmetric random variables such that $X_1$ has a smooth density $f:\R\to\R_+$ and let $G$ be an independent Gaussian random variable with the same variance as $X_2$. A straightforward differentiation shows that the inequality
$$\ent(X_1+\e X_2) \leq \ent(X_1+\e G)$$
as $\e\to0^+$ is equivalent to the comparison of the fourth order Taylor coefficients of these expressions, namely
$$\mb{E}X_2^4 \int_\R f^{(4)}(x)\log f(x)\diff x \geq \mb{E}G^4 \int_\R f^{(4)}(x) \log f(x)\diff x.$$
However, this inequality can easily be seen to be wrong, e.g. by taking $X_1$ to have density function $f(x)=\frac{x^2}{\sqrt{2\pi}}e^{-x^2/2}$ and $X_2$ to be uniformly distributed on a symmetric interval.
\end{remark}

\end{section}


\begin{section}{The B-inequality} \label{sec:4}

We start by establishing a straightforward representation for products of laws of Gaussian mixtures. Let $X_1,\ldots,X_n$ be independent Gaussian mixtures (not necessarily identically distributed) so that $X_i$ has the same distribution as the product $Y_iZ_i$, where $Y_1,\ldots,Y_n$ are independent positive random variables and $Z_1,\ldots,Z_n$ are independent standard Gaussian random variables. Denote by $\nu_i$ the law of $Y_i$, by $\mu_i$ the law of $X_i$ and by $\nu,\mu$ the product measures $\nu_1\times\cdots\times\nu_n$ and $\mu_1\times\cdots\times\mu_n$ respectively. Then, for a Borel set $A\subseteq\R^n$ we have
\begin{equation} \label{eq:measuremixt}
\begin{split}
\mu(A) & = \mb{P}((X_1,\ldots,X_n)\in A) = \mb{P}((Y_1Z_1,\ldots,Y_nZ_n)\in A) \\ & = \int_0^\infty \cdots \int_0^\infty \mb{P}((y_1Z_1,\ldots,y_nZ_n)\in A) \diff\nu_1(y_1)\cdots\diff\nu_n(y_n) \\ & = \int_{(0,\infty)^n} \gamma_n (\Delta(y_1,\ldots,y_n)^{-1} A ) \diff\nu(y_1,\ldots,y_n),
\end{split}
\end{equation}
where $\Delta(y_1,\ldots,y_n)$ is the diagonal matrix with entries $y_1,\ldots,y_n$. In other words, $\mu$ is an average of centered Gaussian measures on $\R^n$. We now proceed with the proof of the B-inequality for Gaussian mixtures.

\medskip

\noindent {\it Proof of Theorem \ref{thm:Bineqmixt}.} Let $X_1,\ldots,X_n$ be as in the statement of the theorem and denote by $h_i$ the density of $Y_i$. Clearly, the log-concavity of the random variable $\log Y_i$ is equivalent to the log-concavity of the function $s\mapsto h_i(e^{-s})$ on $\R$. Let $K\subseteq\R^n$ be a symmetric convex set and $(t_1,\ldots,t_n)\in\R^n$. Then, by \eqref{eq:measuremixt} and the change of variables $y_i=e^{-s_i}$ we have
\begin{equation} \label{eq:measurebineq}
\begin{split}
\mu(\Delta(e^{t_1},\ldots,e^{t_n})K) & = \int_{(0,\infty)^n} \gamma_n (\Delta(y_1^{-1}e^{t_1},\ldots,y_n^{-1}e^{t_n})K) h_1(y_1)\cdots h_n(y_n) \diff y \\ & = \int_{\R^n} \gamma_n(\Delta(e^{s_1+t_1},\ldots,e^{s_n+t_n})K) h_1(e^{-s_1})\cdots h_n(e^{-s_n}) e^{-\sum_{i=1}^n s_i} \diff s.
\end{split}
\end{equation}
The B-inequality for Gaussian measure (Theorem \ref{thm:Bineq}) immediately implies that the function
$$\R^n\times\R^n \ni (s,t) \longmapsto \gamma_n(\Delta(e^{s_1+t_1},\ldots,e^{s_n+t_n})K)$$
is log-concave on $\R^n\times\R^n$. Consequently, the integrand in \eqref{eq:measurebineq} is a log-concave function of $(s,t)\in\R^n\times\R^n$ as a product of log-concave functions. The result now follows from the Pr\'ekopa-Leindler inequality (see, e.g., \cite[Theorem~1.4.1]{AGM15}) which implies that marginals of log-concave functions are log-concave (see also \cite[Theorem~3.15]{GNT14}).
\hfill$\Box$

\smallskip

\begin{remark}
An inspection of the proof of Theorem \ref{thm:Bineqmixt} shows that the same argument also yields the B-inequality for rotationally invariant measures of the form $\diff \mu(x) = f(\|x\|_2) \diff x$, where $f$ is proportional to the density of a Gaussian mixture that satisfies the assumption of Theorem \ref{thm:Bineqmixt}.
\end{remark}

Checking whether a particular Gaussian mixture $X$ satisfies the assumption of Theorem \ref{thm:Bineqmixt} might be non-trivial, since one has to know the distribution of the positive factor $Y$ occurring in its representation. However, by Lemma \ref{lemma:computations}, we know this factor for random variables with densities proportional to $e^{-|t|^p}$ and for symmetric $p$-stable random variables, where $p\in(0,2)$. This allows us to determine the values of $p\in(0,2)$ for which the assumption is satisfied, for each of these random variables. 

To this end, denote, as before, by $g_\alpha$ the density of a standard positive $\alpha$-stable random variable, $\alpha\in(0,1)$. Recall that the positive factor in the representation of a standard symmetric $p$-stable random variable is $(2W_{p/2})^{1/2}$, where $W_{p/2}$ is a standard positive $p/2$-stable random variable. Thus, the assumption of Theorem \ref{thm:Bineqmixt} is equivalent to the log-concavity of the function $s\mapsto g_{p/2}(e^{-s})$ on $\R$. On the other hand, the corresponding factor in the representation of the random variable with density $c_p e^{-|t|^p}$ is of the form $(2V_{p/2})^{-1/2}$ where $V_{p/2}$ has density proportional to $t^{-1/2} g_{p/2}(t)$. Therefore, the corresponding assumption in this case is again equivalent to the log-concavity of $s\mapsto g_{p/2}(e^{-s})$ on $\R$, since the remaining factor $e^{s/2}$ is log-affine. If $X$ is a random variable with density $g:\R\to\R_+$, the log-concavity of $s\mapsto g(e^{-s})$ is referred in the literature as multiplicative strong unimodality of $X$. The multiplicative strong unimodality of positive $\alpha$-stable distributions has been studied by Simon in \cite{Sim11}, who proved that such a random variable has this property if and only if $\alpha\leq1/2$. Combining this with the above observations and Theorem \ref{thm:Bineqmixt} we deduce the following.

\begin{corollary} \label{cor:Bineqexamples}
For every $p\in(0,1]$ the product measure on $\R^n$ with density proportional to $e^{-\|x\|_p^p}$ and the symmetric $p$-stable product measure on $\R^n$ satisfy the $B$-inequality for every symmetric convex set $K\subseteq\R^n$.
\end{corollary}

We now turn to the proof of the small ball estimate for the symmetric exponential measure (Corollary \ref{cor:smallball}) described in the introduction. The argument is very similar to the one in \cite{LO05}.

\medskip

\noindent {\it Proof of Corollary \ref{cor:smallball}.} Let $K\subseteq\R^n$ be a symmetric convex set such that $\mu_1^n(K)\leq1/2$ and we denote by $r=r(K)$ the inradius of $K$. For a set $A\subseteq\R^n$ and $h>0$ we also denote by $A_h$ the \mbox{$h$-enlargement} of $A$, that is, $A_h = A + hB_2^n$. Notice that for $s\in(0,1)$ we have $(sK) \cap (K^c)_{(1-s)r} = \emptyset$, where $K^c$ is the complement of $K$, and thus
\begin{equation} \label{eq:complementsmallball}
\mu_1^n(sK) \leq 1- \mu_1^n ( (K^c)_{(1-s)r}).
\end{equation}
Now, choose $u\geq0$ such that $\mu_1^n(K) = \mu_1((u,\infty))$ or, equivalently, $\mu_1^n(K^c) = \mu_1((-u,\infty))$. Bobkov and Houdr\'e proved in \cite{BH97} that if $A\subseteq\R^n$ is a Borel set and $x\in\R$ is such that $\mu_1^n(A) = \mu_1((x,\infty))$, then for every $h>0$ we have
\begin{equation} \label{eq:bobkovhoudre}
\mu_1^n(A_h) \geq \mu_1 \Big(\Big(x-\frac{h}{2\sqrt{6}},\infty\Big)\Big).
\end{equation}
Combining \eqref{eq:complementsmallball} and \eqref{eq:bobkovhoudre} we get
\begin{equation} \label{eq:sKsmallball}
\mu_1^n(sK) \leq 1- \mu_1\Big(\Big(-u-\frac{(1-s)r}{2\sqrt{6}},\infty\Big)\Big) = \mu_1\Big(\Big(u+\frac{(1-s)r}{2\sqrt{6}},\infty\Big)\Big) =  e^{\frac{s-1}{2\sqrt{6}} r(K)} \mu_1^n(K).
\end{equation}
For $0<t\leq s\leq1$ we can write $s = t^{\frac{\log s}{\log t}}$ and the B-inequality for $\mu_1^n$ implies that
$$\mu_1^n(tK)^{\frac{\log s}{\log t}} \mu_1^n(K)^{1-\frac{\log s}{\log t}} \leq \mu_1^n(sK),$$
or equivalently
\begin{equation} \label{eq:Bineqsmallball}
\frac{\mu_1^n(tK)}{\mu_1^n(K)} \leq \left( \frac{\mu_1^n(sK)}{\mu_1^n(K)} \right)^{\frac{\log t}{\log s}},
\end{equation}
which, in view of \eqref{eq:sKsmallball}, gives the estimate
\begin{equation*}
\mu_1^n(tK) \leq e^{\frac{s-1}{2\sqrt{6}} \cdot \frac{\log t}{\log s} r(K)} \mu_1^n(K) = t^{\frac{r(K)}{2\sqrt{6}} \cdot \frac{s-1}{\log s}} \mu_1^n(K).
\end{equation*}
Taking the limit $s\to1^-$ we finally deduce that
$$\mu_1^n(tK) \leq t^{\frac{r(K)}{2\sqrt{6}}} \mu_1^n(K),$$
for every $t\in[0,1]$, which concludes the proof.
\hfill$\Box$

\begin{remark} \label{rem:pv}
In \cite{PV16a}, Paouris and Valettas proved a different small ball probability estimate for the symmetric exponential measure and any unconditional convex body $K$ in terms of the global parameter $\beta(K) = \mathrm{Var}\|W\|_K/m(K)^2$, where $W$ is distributed according to $\mu_1^n$ and $m(K)$ is the median of $\|\cdot\|_K$ with respect to $\mu_1^n$. Their result is in the spirit of the work \cite{KV07} of Klartag and Vershynin. In the follow-up paper \cite{PV16b}, they showed that a similar estimate holds for every unconditional log-concave measure and unconditional convex body $K$ with a worse dependence on $\beta(K)$. In the particular case of the symmetric exponential measure the unconditionality assumption in the suboptimal estimate from \cite{PV16b} can be omitted, because of Corollary \ref{cor:Bineqexamples}.
\end{remark}

We would like to remark that Theorem \ref{thm:Bineqmixt} combined with a result of Marsiglietti, \cite[Proposition~3.1]{Mar16}, immediately implies the following corollary.

\begin{corollary} \label{cor:marsiglietti}
Let $\mu$ be as in Theorem \ref{thm:Bineqmixt}. Then, for every symmetric convex set $K\subseteq\R^n$ the function $t\mapsto \mu(tK)$ is $\frac{1}{n}$-concave for $t>0$, that is
\begin{equation} \label{eq:marsiglietti}
\mu\big( (\lambda t + (1-\lambda) s) K\big)^{1/n} \geq \lambda \mu(tK)^{1/n} + (1-\lambda) \mu(sK)^{1/n},
\end{equation}
for every $t,s>0$ and $\lambda\in(0,1)$.
\end{corollary}

\end{section}


\begin{section}{Correlation inequalities} \label{sec:5}

To prove the correlation inequality for Gaussian mixtures (Theorem \ref{thm:correlation}) we will use Royen's Gaussian correlation inequality (Theorem \ref{thm:Gausscorr}), along with a simple lemma for the standard Gaussian measure. Recall that we write $\Delta(y)=\Delta(y_1,\ldots,y_n)$ for the diagonal $n\times n$ matrix with diagonal $y=(y_1,\ldots,y_n)$.

\begin{lemma} \label{lemma:monotonicitycorrelation}
Let $K$ be a symmetric convex set in $\R^n$. Then the function $t\mapsto \gamma_n\big(\Delta(t,1,\ldots,1)K\big)$ is nondecreasing for $t>0$.
\end{lemma}

\begin{proof}
It clearly suffices to consider the case when $K$ has nonempty interior. We will prove that the function $\psi(x)=\log\gamma_n\big(\Delta(e^{x},1,\ldots,1)K\big)$ is nondecreasing on the real line. By virtue of the B-inequality for the standard Gaussian measure (Theorem \ref{thm:Bineq}), $\psi$ is concave. To verify that $\psi$ is nondecreasing, it is enough to prove that $\lim_{x\to\infty} \psi(x) >-\infty$. Take $\delta>0$ such that $[-\delta,\delta]^n\subseteq K$. For every real number $x$ we have
$$\psi(x) \geq \log\gamma_n([-e^x\delta,e^x\delta]\times[-\delta,\delta]^{n-1}),$$
which, for $x\to\infty$, gives
$$\lim_{x\to\infty} \psi(x) \geq \log\gamma_n(\mathbb{R}\times[-\delta,\delta]^{n-1}) > -\infty.$$
This concludes the proof of the lemma.
\end{proof}

\noindent {\it Proof of Theorem \ref{thm:correlation}.} Let $\mu$ be a product of laws of Gaussian mixtures. According to \eqref{eq:measuremixt} for every Borel set $A\subseteq\R^n$ we have
$$\mu(A) = \int_{(0,\infty)^n}\gamma_n (\Delta(y)^{-1} A) \diff\nu_1(y_1)\cdots\diff\nu_n(y_n),$$
for some probability measures $\nu_1,\ldots,\nu_n$ on $(0,\infty)$. Let $K,L\subseteq\R^n$ be symmetric convex sets. The Gaussian correlation inequality yields
\begin{equation} \label{eq:useroyen}
\begin{split} 
\mu(K\cap L) & = \int_{(0,\infty)^n} \gamma_n (\Delta(y)^{-1} K\cap \Delta(y)^{-1}L) \diff\nu_1(y_1)\cdots\diff\nu_n(y_n) \\ & \geq \int_{(0,\infty)^n} \gamma_n (\Delta(y)^{-1}K) \gamma_n(\Delta(y)^{-1}L) \diff\nu_1(y_1)\cdots\diff\nu_n(y_n).
\end{split}
\end{equation}
Fix $y_1,\ldots,y_{n-1}>0$. Lemma \ref{lemma:monotonicitycorrelation} implies that the functions $y_n \mapsto \gamma_n(\Delta(y)^{-1}K)$ and $y_n\mapsto \gamma_n(\Delta(y)^{-1}L)$ are nonincreasing on $(0,\infty)$. Consequently, combining \eqref{eq:useroyen} and Chebyshev's integral inequality (see, e.g., \cite[p.~168]{HLP88}) for the probability measure $\nu_n$, we get
\begin{equation*}
\begin{split}
\mu   (K & \cap L) \geq \\ & \int_{(0,\infty)^{n-1}} \Big( \int_0^\infty\gamma_n(\Delta(y)^{-1}K) \diff\nu_n(y_n)\Big) \Big(\int_0^\infty \gamma_n(\Delta(y)^{-1}L) \diff\nu_n(y_n) \Big) \diff\nu_1(y_1)\cdots\diff\nu_{n-1}(y_{n-1}).
\end{split}
\end{equation*}
After iteratively applying Chebyshev's inequality to $\nu_1,\ldots,\nu_{n-1}$ we finally deduce that
\begin{equation*}
\begin{split}
\mu(K\cap L) & \geq \int_{(0,\infty)^n} \gamma_n (\Delta(y)^{-1}K)  \diff\nu_1(y_1)\cdots \diff\nu_n(y_n) \cdot \int_{(0,\infty)^n} \gamma_n (\Delta(y)^{-1}L) \diff\nu_1(y_1)\cdots \diff\nu_n(y_n) \\ & = \mu(K) \mu(L),
\end{split}
\end{equation*}
which is the correlation inequality \eqref{eq:correlationmixt}.
\hfill$\Box$

\begin{remark} \label{remark:correlation}
Similarly to the B-inequality, an inspection of the proof of Theorem \ref{thm:correlation} reveals that the same argument also gives the correlation inequality for rotationally invariant probability measures of the form $\diff\mu(x) = f(\|x\|_2)\diff x$, where $f$ is proportional to the density of a Gaussian mixture.
\end{remark}

Recall that a function $f:\R^n\to\R_+$ is called quasiconcave if for any $t\geq0$ the set $A_t=\{x\in\R^n: \ f(x)\geq t\}$ is convex. Writing
$$f(x) = \int_0^\infty {\bf 1}_{A_t} (x) \diff t, \ \ \ x\in\R,$$
one can immediately see that if a measure $\mu$ satisfies the correlation inequality \eqref{eq:correlationmixt} for any symmetric convex sets $K,L\subseteq\R^n$ then for every symmetric quasiconcave functions $f,g:\R^n\to\R_+$ we have
\begin{equation} \label{eq:quasiconcave}
\int_{\R^n} f(x) g(x) \diff\mu(x) \geq \int_{\R^n} f(x)\diff\mu(x) \cdot \int_{\R^n} g(x)\diff\mu(x).
\end{equation}
Correlation inequalities of the form \eqref{eq:quasiconcave} were treated by Koldobsky and Montgomery-Smith in \cite{KMS96} for another class of functions when $\mu$ is a general symmetric stable measure on $\R^n$. Recall that the law $\mu$ of a random vector $X$ in $\R^n$ is called a symmetric $p$-stable measure if every marginal $\langle X,a\rangle$, $a\in\R^n$, is a symmetric $p$-stable random variable. It is a well known fact (see, e.g., \cite[p.~312]{Wer84}) that symmetric $p$-stable random vectors $X=(X_1,\ldots,X_n)$ in $\R^n$ are in one-to-one correspondence with finite measures $m_X$ on the unit sphere $S^{n-1}$ such that
\begin{equation} \label{eq:stablemeasurerepresentation}
\mb{E} \exp\Big(i\sum_{j=1}^n a_j X_j\Big) = \exp\Big(-\int_{S^{n-1}} \Big| \sum_{j=1}^n a_j x_j \Big|^p \diff m_X(x)\Big),
\end{equation}
 for every $a_1,\ldots,a_n\in\R$. We will argue that the correlation inequality \eqref{eq:correlationmixt} holds for the law $\mu$ of any symmetric $p$-stable random vector $X$ in $\R^n$. Assume first that the corresponding measure $m_X$ on $S^{n-1}$ has a finite support, namely $\mathrm{supp}(m_X) = \{y_1,\ldots,y_\ell\}$, and let $Y$ be a standard $\ell$-dimensional symmetric $p$-stable random vector with independent coordinates. In this case, one can find $\theta_1,\ldots,\theta_n\in\R^\ell$ such that $X_j$ has the same distribution as $\langle Y,\theta_j\rangle$ or, in other words, $X$ is a linear image of $Y$ and the correlation inequality \eqref{eq:correlationmixt} immediately follows. For a general measure $m_X$ on $S^{n-1}$ there exists a sequence of finitely supported measures $m_\ell$ that converges to $m_X$ in the weak* topology (e.g. by the Krein-Milman theorem) which means, by \eqref{eq:stablemeasurerepresentation}, that the corresponding $p$-stable random vectors $X_\ell$ converge to $X$ in distribution. Note that to prove the correlation inequality \eqref{eq:correlationmixt} for a symmetric $p$-stable measure $\mu$ on $\R^n$, it suffices to consider the case when $K,L\subseteq\R^n$ are convex polytopes, which are sets whose boundaries are contained in a finite union of affine hyperplanes. However, any affine hyperplane is of $\mu$-measure zero, since the one-dimensional marginals of $\mu$ are $p$-stable, thus continuous. Therefore, the convergence in distribution concludes the proof of the following corollary.

\begin{corollary} \label{cor:corrstable}
Let $\mu$ be a symmetric $p$-stable measure on $\R^n$. Then for every symmetric convex sets $K,L\subseteq\R^n$ we have
$$\mu(K\cap L) \geq \mu(K) \mu(L).$$
\end{corollary}

This corollary implies inequalities of the form \eqref{eq:quasiconcave}, analogous to the ones proven in \cite{KMS96}. It also implies that the multivariate Cauchy distribution, defined as $\diff\mu(x) = c_n (1+\|x\|_2^2)^{-\frac{n+1}{2}}\diff x$, satisfies the correlation inequality \eqref{eq:correlationmixt}. Notice that this also follows from Remark \ref{remark:correlation}. In \cite{Mem15}, the author showed that this is actually equivalent to Corollary \ref{cor:corrsphere}. We reproduce his argument below.

\medskip

\noindent {\it Proof of Corollary \ref{cor:corrsphere}.}
Consider the hyperplane $\R^{n-1} \equiv \R^{n-1}\times\{0\}\subseteq\R^n$ and let $S\subseteq\R^n$ be the sphere of radius 1 centered at $e_n=(0,\ldots,0,1)$. Denote by $S_+$ the open lower hemisphere of $S$, i.e. $S_+=\{x\in S: \ x_n<1\}$, and define a bijection $q:S_+\to\R^{n-1}$ by the formula
\begin{equation} \label{eq:gnomonic}
q(x) = \mbox{the point of } \R^{n-1} \ \mbox{which lies on the line joining } x \mbox{ to } e_n.
\end{equation}
One can easily check that closed arcs of great circles on $S_+$ are mapped to line segments on $\R^{n-1}$ and vice versa, which immediately implies that geodesically convex sets in $S_+$ are in one-to-one correspondence with convex sets in $\R^{n-1}$. Moreover, since $q(0)=0$, symmetry in $\R^{n-1}$ agrees with geodesic symmetry in $S_+$. Denoting by $\mu$ the push-forward under $q$ of the normalized surface area measure on $S_+$, we get that for every $r>0$, $\mu$ satisfies the identity
$$\mu(rB_2^{n-1}) = \frac{|B_{S^{n-1}}(\arctan r)|}{|S_+|},$$
where $B_{S^{n-1}}(\theta)$ is a spherical cap of radius $\theta$ on $S^{n-1}$. A simple computation for the volume of spherical caps along with the rotational invariance of $\mu$ shows that $\mu$ is precisely the law of the multivariate Cauchy distribution on the hyperplane $\R^{n-1}$. Therefore, for two symmetric geodesically convex sets $K,L\subseteq S_+$, the multivariate Cauchy correlation inequality for the symmetric convex sets $q(K),q(L)\subseteq\R^{n-1}$ implies that
$$\frac{|K\cap L|}{|S_+|} = \mu(q(K\cap L)) = \mu(q(K)\cap q(L)) \geq \mu(q(K))\cdot\mu(q(L)) = \frac{|K|}{|S_+|} \cdot \frac{|L|}{|S_+|},$$
which completes the proof of the corollary.
\hfill$\Box$

\begin{remark} \label{rem:correlationforp>2}
It is a straightforward consequence of Theorem \ref{thm:correlation} that the product probability measure $\mu_p^n$ with density $c_p^ne^{-\|x\|_p^p}$ satisfies the correlation inequality \eqref{eq:correlationmixt} for every $p\in(0,2]$ and $n\geq1$. It turns out that this is the exact range of $p>0$ for which this property holds. To see this, take $\delta>0$ and consider the symmetric strips
$$K_\delta = \{(x,y)\in\R^2: \ |x-y| \leq \delta\} \ \ \mbox{and} \ \ L_\delta=\{(x,y)\in\R^2: \ |x+y|\leq \delta\}$$
on the plane. We will show that $\mu_p^2(K_\delta\cap L_\delta) < \mu_p^2(K_\delta)\mu_p^2(L_\delta)$ for $p>2$ and small enough $\delta>0$. Indeed, a straightforward differentiation yields that the Taylor expansions of these two quantities around $\delta=0$ are
$$\mu_p^2(K_\delta\cap L_\delta) = 4c_p^2 \delta^2 + o(\delta^3) \ \ \mbox{and} \ \ \mu_p^2(K_\delta) \mu_p^2(L_\delta) = 4c_p^2 2^{1-\frac{2}{p}} \delta^2  + o(\delta^3)$$
and since $1<2^{1-\frac{2}{p}}$ for $p>2$, the correlation inequality \eqref{eq:correlationmixt} cannot hold for small enough $\delta>0$. A computation along the same lines together with Remark \ref{remark:correlation} prove that a similar behavior is exhibited by the rotationally invariant probability measures with densities proportional to $e^{-\|x\|_2^p}$: they satisfy \eqref{eq:correlationmixt} if and only if $p\in(0,2]$.
\end{remark}

\begin{remark}
After the submission of this paper, we learned from J. Zinn about the works \cite{LP99} and \cite{LP03} of Lewis and Pritchard on measures supporting correlation inequalities \eqref{eq:correlationmixt}. In \cite{LP99}, the authors proved, among other things, that the uniform probability measure on any convex body $K$ in $\R^n$ does not support a correlation inequality, even though we now know that this holds for the uniform measure on a hemisphere, where usual convexity and symmetry are replaced by geodesic convexity and symmetry (Corollary \ref{cor:corrsphere}). In \cite{LP03}, they showed that any rotationally invariant measure $\mu$ on $\R^n$ which supports a correlation inequality must satisfy
$$\int_{\R^n} e^{a\|x\|_2^2}\ \dd \mu(x) = \infty$$
for some constant $a>0$. The computation mentioned in the last sentence of Remark \ref{rem:correlationforp>2} is also a consequence of their result.
We are grateful to J. Zinn for providing us these references.
\end{remark}

\end{section}


\begin{section}{Sections and projections of $B_q^n$ revisited} \label{sec:6}

In this section we derive the comparison results for geometric parameters of hyperplane sections and projections of the balls $B_q^n$ described in the introduction. First, let us explain how the comparison of the aforementioned Gaussian parameters (Theorem \ref{thm:laplacegaussian}) implies the comparison of volume (Corollary \ref{cor:sections}) and mean width (Corollary \ref{cor:width}), following \cite{BGMN05}.

\medskip

\noindent {\it Proof of Corollaries \ref{cor:sections} and \ref{cor:width}.} Fix $q\in(0,2)$ and let $a=(a_1,\ldots,a_n), b=(b_1,\ldots,b_n)\in\R^n$ be unit vectors such that $(a_1^2,\ldots,a_n^2) \preceq (b_1^2,\ldots,b_n^2)$. Recall that $G_a, G_b$ are standard Gaussian random vectors on the hyperplanes $a^\perp$ and $b^\perp$ respectively. According to Theorem \ref{thm:laplacegaussian}, for every $\lambda>0$ we have
$$\mb{E} e^{-\lambda \|G_a\|_{B_q^n\cap a^\perp}^q} \leq \mb{E} e^{-\lambda\|G_b\|_{B_q^n\cap b^\perp}^q}.$$
Integrating this inequality with respect to $\lambda$ and any measure $\mu$ on $(0,\infty)$ we deduce that
$$\mb{E} \int_0^\infty  e^{-\lambda \|G_a\|_{B_q^n\cap a^\perp}^q} \diff\mu(\lambda) \leq \mb{E} \int_0^\infty e^{-\lambda\|G_b\|_{B_q^n\cap b^\perp}^q} \diff\mu(\lambda),$$
which, by Bernstein's theorem, is equivalent to the validity of the inequality
\begin{equation} \label{eq:completelymonotonecomparison}
\mb{E} g(\|G_a\|_{B_q^n\cap a^\perp}^q) \leq \mb{E} g(\|G_b\|_{B_q^n\cap b^\perp}^q)
\end{equation}
 for every completely monotonic function $g:(0,\infty)\to\R$. In particular, choosing  $g(s) = s^{-\alpha/q}$, we get that
$$\mb{E} \|G_a\|_{B_q^n\cap a^\perp}^{-\alpha} \leq \mb{E} \|G_b\|_{B_q^n \cap b^\perp}^{-\alpha},$$
provided that $0<\alpha<n-1$ so that the integrals are finite. Integration in polar coordinates now shows that for every $0<\alpha<n-1$ we have
\begin{equation} \label{eq:spherepowers}
\int_{S(a^\perp)} \|\theta\|_{B_q^n\cap a^\perp}^{-\alpha} \diff\sigma_a(\theta) \leq \int_{S(b^\perp)} \|\theta\|_{B_q^n\cap b^\perp}^{-\alpha} \diff\sigma_b(\theta),
\end{equation}
where $\sigma_a, \sigma_b$ are the rotationally invariant probability measures on the unit spheres $S(a^\perp), S(b^\perp)$ of the hyperplanes $a^\perp$ and $b^\perp$, respectively. Letting $\alpha\to n-1$ in \eqref{eq:spherepowers} along with the identity
\begin{equation} \label{eq:volumeradius}
\int_{S^{m-1}} \|\theta\|_{K}^{-m} \diff\sigma(\theta) = \frac{|K|}{|B_2^m|},
\end{equation}
which holds for every symmetric convex body $K$ in $\R^m$, imply that $|B_q^n \cap a^\perp| \leq |B_q^n \cap b^\perp|$ and Corollary \ref{cor:sections} follows.

Furthermore, applying \eqref{eq:completelymonotonecomparison} to $g(s) = e^{-\lambda s^\beta}$, where $\beta\in(0,1]$ and $\lambda>0$, we have
$$\mb{E} e^{-\lambda\|G_a\|_{B_q^n\cap a^\perp}^{\beta q}} \leq \mb{E} e^{-\lambda\|G_b\|_{B_q^n\cap b^\perp}^{\beta q}}.$$
Since both sides, as functions of $\lambda>0$, are equal at $\lambda=0$ we deduce that their derivatives at $\lambda=0$ also satisfy the same inequality, that is
$$\mb{E} \|G_b\|_{B_q^n\cap b^\perp}^{\beta q} \leq \mb{E} \|G_a\|_{B_q^n\cap a^\perp}^{\beta q},$$
for every $\beta\in(0,1]$. Choosing $\beta=1/q$ and integrating in polar coordinates yields
\begin{equation} \label{eq:Mcomparison}
\int_{S(b^\perp)} \|\theta\|_{B_q^n\cap b^\perp} \diff\sigma_b(\theta) \leq \int_{S(a^\perp)} \|\theta\|_{B_q^n\cap a^\perp} \diff\sigma_a(\theta).
\end{equation}
Recall that for a symmetric convex body $K$ in $\R^m$, the polar body $K^o$ of $K$ is defined to be $K^o=\{x\in\R^m: \ \langle x,y\rangle \leq 1 \ \mbox{for every } y\in K\}$ and if $\|\cdot\|_{K^o}$ is the norm associated with $K^o$ then $\|\theta\|_{K^o} = h_K(\theta)$ for every $\theta\in S^{n-1}$. Thus, combining \eqref{eq:Mcomparison} with the well known identity $B_q^n\cap H = (\proj_H(B_{q^*}^n))^o$, where $q^*>2$ is such that $\frac{1}{q}+\frac{1}{q^*}=1$, we deduce that
\begin{equation}
w(\proj_{b^\perp}(B_{q^*}^n)) \leq w(\proj_{a^\perp}(B_{q^*}^n)).
\end{equation}
In particular, for every hyperplane $H\subseteq\R^n$ we obtain
$$w(\proj_{H_1}(B_{q^*}^n)) \leq w(\proj_H(B_{q^*}^n)) \leq w(\proj_{H_n}(B_{q^*}^n)),$$
where $H_1=(1,0,\ldots,0)^\perp$ and $H_n=(1,\ldots,1)^\perp$. This concludes the proof of Corollary \ref{cor:width}. \hfill$\Box$

\medskip

We finally proceed with the proof of Theorem \ref{thm:laplacegaussian}.

\medskip

\noindent {\it Proof of Theorem \ref{thm:laplacegaussian}.} Fix $q\in(0,2)$. For a hyperplane $H=a^\perp$, where $a=(a_1,\ldots,a_n)\in\R^n$ is a unit vector, let $G_a$ be a standard Gaussian random vector on $a^\perp$ and denote by $H(\e)$ the set
\begin{equation} \label{eq:h(e)}
H(\e) = \{x\in\R^n: \ |\langle x,a\rangle| <\e\}.
\end{equation} 
 To proceed, we will need a representation from \cite[Lemma~14]{BGMN05} for the Laplace transforms of $\|G_a\|_{B_q^n\cap H}^q$ that reads as follows. For every $\lambda>0$ there exist constants $\alpha(q,\lambda), \beta(q,\lambda)>0$ and $c(q,\lambda,n)>0$ such that for every hyperplane $H=a^\perp$, $a=(a_1,\ldots,a_n)\in S^{n-1}$, we have
\begin{equation} \label{eq:laplacerepresentation}
\mb{E} e^{-\lambda \|G_a\|_{B_q^n\cap H}^q} = c(q,\lambda,n) \lim_{\e\to0^+} \frac{1}{2\e} \mu_{q,\lambda}^n (H(\e)),
\end{equation}
where the probability measure $\mu_{q,\lambda}$ on $\R$ is of the form
\begin{equation} \label{eq:measurelaplace}
\diff\mu_{q,\lambda}(t) = e^{-\alpha(q,\lambda) |t|^q - \beta(q,\lambda)t^2} \diff t
\end{equation}
and $\mu_{q,\lambda}^n = \mu_{q,\lambda}^{\otimes n}$. An immediate application of Theorem \ref{thm:charact} yields that $\mu_{q,\lambda}$ is the law of a Gaussian mixture. Thus, by \eqref{eq:measuremixt} there exists a probability measure $\nu=\nu(q,\lambda)$ on $(0,\infty)$ such that if $A\subseteq\R^n$ is a Borel set, then
$$\mu_{q,\lambda}^n(A) = \int_{(0,\infty)^n} \gamma_n(\Delta(y)^{-1}A) \diff\nu_n(y),$$
where $\nu_n = \nu^{\otimes n}$. Notice that for the symmetric strip \eqref{eq:h(e)} we have
$$\Delta(y)^{-1} H(\e) = \Big\{x\in\R^n: \Big| \sum_{j=1}^n a_jy_jx_j \Big|<\e\Big\},$$
that is, $\Delta(y)^{-1}H(\e)$ is also a symmetric strip of width $\big(\sum_{j=1}^n a_j^2 y_j^2 \big)^{-1/2}\e$. Consequently, the rotational invariance of the Gaussian measure implies that
\begin{equation} \label{eq:measurelaplacestrip}
\mu_{q,\lambda}^n(H(\e)) = 2 \int_{(0,\infty)^n} \Psi\Big(\Big(\sum_{j=1}^n a_j^2 y_j^2\Big)^{-1/2}\e\Big)\diff\nu_n(y),
\end{equation}
where $\Psi(s) = \frac{1}{\sqrt{2\pi}} \int_0^s e^{-x^2/2}\diff x$. Combining \eqref{eq:laplacerepresentation} and \eqref{eq:measurelaplacestrip} we deduce that
\begin{equation*}
\begin{split}
c(q,\lambda,n)^{-1}\cdot \mb{E} e^{-\lambda \|G_a\|_{B_q^n\cap H}^q} & = \lim_{\e\to0^+} \frac{1}{\e} \int_{(0,\infty)^n} \Psi \Big( \Big( \sum_{j=1}^n a_j^2 y_j^2 \Big)^{-1/2} \e \Big) \diff\nu_n(y)  \\ & = \int_{(0,\infty)^n
} \frac{\diff}{\diff\e} \Big|_{\e=0} \Psi\Big(\Big(\sum_{j=1}^n a_j^2 y_j^2\Big)^{-1/2}\e\Big) \diff\nu_n(y) \\ & =\frac{1}{\sqrt{2\pi}} \int_{(0,\infty)^n} \Big(\sum_{j=1}^n a_j^2 y_j^2 \Big)^{-1/2} \diff\nu_n(y) =\frac{1}{\sqrt{2\pi}} \mb{E} \Big( \sum_{j=1}^n a_j^2 Y_j^2\Big)^{-1/2},
\end{split}
\end{equation*}
where $Y_1,\ldots,Y_n$ are i.i.d. random variables distributed according to $\nu$. To verify the assumptions of the dominated convergence theorem for the swap of the limit and integration in the second equality, it suffices to check that $\big(\sum_{j=1}^n a_j^2 y_j^2\big)^{-1/2} \in L_1(\nu_n)$, since $\Psi(s) \leq \frac{s}{\sqrt{2\pi}}$ for $s>0$. This immediately follows by Fatou's lemma, that is
$$\int_{(0,\infty)^n} \Big(\sum_{j=1}^n a_j^2 y_j^2 \Big)^{-1/2} \diff\nu_n(y) \leq \sqrt{2\pi} \liminf_{\e\to0^+}\frac{1}{2\e} \mu_{q,\lambda}^n(H(\e)) = \sqrt{2\pi}c(q,\lambda,n)^{-1} \mb{E} e^{-\lambda\|G_a\|_{B_q^n\cap H}^q} <\infty.$$
Now, since $t\mapsto t^{-1/2}$ is a convex function on $(0,\infty)$ and $Y_1,\ldots,Y_n$ are i.i.d. random variables, Marshall and Proschan's result \eqref{eq:marpro} implies the comparison \eqref{eq:laplacegaussiancomp}, as required.
\hfill$\Box$

\medskip

We note that the crucial identity \eqref{eq:measurelaplacestrip} can also be proven in purely probabilistic terms. Let $X_1,\ldots,X_n$ be i.i.d. random variables distributed according to $\mu_{q,\lambda}$ and take i.i.d. positive random variables $Y_1,\ldots,Y_n$ and standard Gaussian random variables $Z_1,\ldots,Z_n$ such that $X_i$ has the same distribution as the product $Y_iZ_i$. Then we have
\begin{equation*}
\begin{split}
\mu_{q,\lambda}^n(H(\e)) & = \mb{E}_Y \mb{P}_Z \Big(\Big|\sum_{j=1}^n a_j Y_jZ_j \Big| <\e \Big) = \mb{E}_Y \mb{P}_Z\Big( |Z| \Big(\sum_{j=1}^n a_j^2 Y_j^2 \Big)^{1/2} <\e\Big) \\ & =\mb{E}_Y \mb{P}_Z \Big( |Z| < \Big(\sum_{j=1}^n a_j^2 Y_j^2\Big)^{-1/2} \e\Big) = 2\mb{E}  \Big[\Psi \Big(\Big(\sum_{j=1}^n a_j^2 Y_j^2\Big)^{-1/2} \e\Big)\Big],
\end{split}
\end{equation*}
where $Z$ is a standard Gaussian random variable, independent of the $Y_i$.

\begin{remark}
A similar approach also yields a direct proof of Corollary \ref{cor:sections}. The crucial ingredient in this case would be an identity from \cite{MP88} and \cite{Bar95} instead of \eqref{eq:laplacerepresentation}. It is proven there that for every $q\in(0,\infty)$, there exists a constant $c(q,n)>0$ such that if $H\subseteq\R^n$ is any hyperplane and $H(\e)$ is defined by \eqref{eq:h(e)}, then
\begin{equation} \label{eq:meyerpajorformula}
|B_q^n\cap H| = c(q,n) \lim_{\e\to0^+} \frac{1}{2\e} \mu_q^n(H(\e)),
\end{equation}
where $\mu_q^n$ is the measure on $\R^n$ with density proportional to $e^{-\|x\|_q^q}$. Since this measure is also a product of laws of i.i.d. Gaussian mixtures the preceding argument works identically. 

In \cite{KZ03}, Koldobsky and Zymonopoulou investigated extremal volumes of sections of the complex $\ell_q$-balls $B_q^n(\C)$, which can also be treated by the approach presented above. From now on we will adopt the obvious identification of $\C^n$ with $\R^{2n}$ without further ado. We will denote by $\langle\cdot,\cdot\rangle$ the standard Hermitian inner product on $\C^n$ and for a vector $\zeta\in\C^n$ we will write $\zeta^\perp$ for the complex hyperplane orthogonal to $\zeta$. Recall that for a vector $z=(x_1,y_1,\ldots,x_n,y_n)\in\R^{2n}$ we denote
$$\|z\|_{\ell_q^n(\C)} = \Big(\sum_{j=1}^n (x_j^2+y_j^2)^{q/2} \Big)^{1/q}= \Big( \sum_{j=1}^n |z_j|^q\Big)^{1/q} ,$$
where $z_j = x_j+iy_j$, and $B_q^n(\C) = \{ z\in\R^{2n}: \|z\|_{\ell_q^n(\C)} \leq 1\}$. Let $H_n=\xi^\perp$ be any complex hyperplane such that $|\xi_1|=\cdots=|\xi_n|$ and $H_1=\eta^\perp$ be such that $\eta_j=0$ for $j\geq2$, where $\xi=(\xi_1,\ldots,\xi_n), \eta=(\eta_1,\ldots,\eta_n)\in\C^{n}$. In \cite{KZ03}, the authors proved that for any $q\in(0,2)$ and any complex hyperplane $H\subseteq\C^n$ the inequalities
\begin{equation} \label{eq:koldobskyzymonopoulou}
|B_q^n(\C) \cap H_n| \leq |B_q^n(\C)\cap H|\leq|B_q^n(\C)\cap H_1|
\end{equation}
hold true. We will sketch an alternative proof of their result, similar to the proof of Theorem \ref{thm:laplacegaussian}. For a complex hyperplane $H=\zeta^\perp$, where $\zeta\in\C^n$, and $\e>0$ denote by $H_{\mathrm{cyl}}(\e)$ the cylinder
\begin{equation}
H_{\mathrm{cyl}}(\e) = \{z\in\C^n: \ |\langle z,\zeta\rangle| <\e\}.
\end{equation}
One can prove (see also \cite[Corollary~2.5]{MP88}) that there exists a constant $c(q,n)>0$ such that for every complex hyperplane $H\subseteq\C^n$ we have
\begin{equation} \label{eq:complexballformula}
|B_q^n(\C) \cap H| = c(q,n) \lim_{\e\to0} \frac{1}{\e^2} \tau_q^n(H_{\mathrm{cyl}}(\e)),
\end{equation}
where the measure $\tau_q^n$ on $\R^{2n}$ is of the form
$$\diff\tau_q^n (x,y) = c_q^n e^{-\sum_{j=1}^n (x_j^2+y_j^2)^{q/2}} \diff x\diff y.$$
Writing $e^{-s^{q/2}} = \int_0^\infty e^{-ts}\diff\mu(t)$ for some measure $\mu$, we deduce that the density of $\tau_q^n$ can be written in the form
$$e^{-\sum_{j=1}^n (x_j^2+y_j^2)^{q/2}} = \int_{(0,\infty)^n} e^{-\sum_{j=1}^n t_j (x_j^2+y_j^2)} \diff\mu_n(t),$$
where $\mu_n=\mu^{\otimes n}$. Therefore, an application of Fubini's theorem and a change of variables imply that there exists a measure $\nu$ on $(0,\infty)$ such that for $\nu_n=\nu^{\otimes n}$ and for every Borel set $A\subseteq\R^{2n}$ we can write
\begin{equation} \label{eq:taurepresentation}
\tau_q^n(A) = \int_{(0,\infty)^n} \gamma_{2n} (\Delta(y_1,y_1,\ldots,y_n,y_n)^{-1} A) \diff\nu_n(y),
\end{equation}
where each coordinate of $y=(y_1,\ldots,y_n)$ is repeated twice. Notice that the image 
$$\Delta(y_1,y_1,\ldots,y_n,y_n)^{-1} H_{\mathrm{cyl}}(\e) = \Big\{ z\in\C^n: \Big| \sum_{j=1}^n \zeta_j y_j z_j \Big| <\e\Big\}$$
is still a cylinder in $\C^n$ with radius $\big(\sum_{j=1}^n |\zeta_j|^2 y_j^2\big)^{-1/2}\e$. Thus, the unitary invariance of complex Gaussian measure and a simple calculation in polar coordinates imply that
\begin{equation} \label{eq:measurecylinder}
\gamma_{2n}(\Delta(y_1,y_1,\ldots,y_n,y_n)^{-1} H_{\mathrm{cyl}}(\e)) = 1- \exp\Big(-\frac{1}{2}\Big(\sum_{j=1}^n |\zeta_j|^2 y_j^2\Big)^{-1} \e^2\Big).
\end{equation}
After interchanging limit and integration in \eqref{eq:complexballformula} and using \eqref{eq:taurepresentation}, \eqref{eq:measurecylinder} we deduce that
\begin{equation*}
\begin{split}
|B_q^n(\C)\cap H| =\frac{c(q,n)}{2}\cdot \int_{(0,\infty)^n} \Big(\sum_{j=1}^n |\zeta_j|^2 y_j^2 \Big)^{-1} \diff\nu_n(y)  = \frac{c(q,n)}{2} \cdot \mb{E} \Big( \sum_{j=1}^n |\zeta_j|^2 Y_j^2 \Big)^{-1},
\end{split}
\end{equation*}
where $Y_1,\ldots, Y_n$ are i.i.d. random variables distributed according to $\nu$. This yields \eqref{eq:koldobskyzymonopoulou} as well as a more general comparison result, similar to Corollary \ref{cor:sections}, by a direct application of Marshall and Proschan's result \eqref{eq:marpro}. \hfill$\Box$
\end{remark}

We note that, in view of Ball's theorem from \cite{Bal86}, a Schur monotonicity result for the volume of sections of $B_q^n$ cannot hold in any fixed dimension $n\geq2$ and $q$ large enough. Similarly, according to Szarek's result from \cite{Sza76}, the same can be said for the volume of projections of $B_q^n$ for values close to $q=1$. Finally, we want to stress that a careful look in the previous works \cite{BN02}, \cite{Kol98} and \cite{KZ03} reveals that, even though not stated explicitly, the Schur monotonicity for the volume was established there as well. The new aspect here is the replacement of representations which were Fourier-analytic in flavor by others that exploit the rotational invariance of the Gaussian measure.

\medskip

\noindent {\it Remark.} Following the publication of this article, Gaussian mixtures have also been used for geometric purposes in \cite{Esk18} and \cite{NT18}.

\medskip

\noindent {\bf Acknowledgements.} We would like to thank Apostolos Giannopoulos for providing several useful references and Petros Valettas for his comments regarding Remark \ref{rem:pv}. We are also indebted to Christian Houdr\'e and Joel Zinn for communicating the references \cite{AH03}, \cite{LP99} and \cite{LP03} to us. Finally, we are very grateful to Assaf Naor for many helpful discussions and to Franck Barthe for valuable feedback.

\end{section}

\bibliography{gaussmixt}

\begin{thebibliography}{AAGM15}

\bibitem[AAGM15]{AGM15}
S.~Artstein-Avidan, A.~Giannopoulos, and V.~D. Milman.
\newblock {\em Asymptotic geometric analysis. {P}art {I}}, volume 202 of {\em
  Mathematical Surveys and Monographs}.
\newblock American Mathematical Society, Providence, RI, 2015.

\bibitem[ABBN04]{ABBN04}
S.~Artstein, K.~M. Ball, F.~Barthe, and A.~Naor.
\newblock Solution of {S}hannon's problem on the monotonicity of entropy.
\newblock {\em J. Amer. Math. Soc.}, 17(4):975--982 (electronic), 2004.

\bibitem[AH03]{AH03}
R.~Averkamp and C.~Houdr\'e.
\newblock Wavelet thresholding for non-necessarily {G}aussian noise: idealism.
\newblock {\em Ann. Statist.}, 31(1):110--151, 2003.

\bibitem[AK01]{AK01}
S.~Arora and R.~Kannan.
\newblock Learning mixtures of arbitrary {G}aussians.
\newblock In {\em Proceedings of the {T}hirty-{T}hird {A}nnual {ACM}
  {S}ymposium on {T}heory of {C}omputing}, pages 247--257. ACM, New York, 2001.

\bibitem[Bal86]{Bal86}
K.~Ball.
\newblock Cube slicing in {${\bf R}^n$}.
\newblock {\em Proc. Amer. Math. Soc.}, 97(3):465--473, 1986.

\bibitem[Bar95]{Bar95}
F.~Barthe.
\newblock Mesures unimodales et sections des boules {$B^n_p$}.
\newblock {\em C. R. Acad. Sci. Paris S\'er. I Math.}, 321(7):865--868, 1995.

\bibitem[BC02]{BC02}
A.~Baernstein, II and R.~C. Culverhouse.
\newblock Majorization of sequences, sharp vector {K}hinchin inequalities, and
  bisubharmonic functions.
\newblock {\em Studia Math.}, 152(3):231--248, 2002.

\bibitem[BC15]{BC15}
S.~G. Bobkov and G.~P. Chistyakov.
\newblock Entropy power inequality for the {R}\'enyi entropy.
\newblock {\em IEEE Trans. Inform. Theory}, 61(2):708--714, 2015.

\bibitem[BGMN05]{BGMN05}
F.~Barthe, O.~Gu{\'e}don, S.~Mendelson, and A.~Naor.
\newblock A probabilistic approach to the geometry of the {$l^n_p$}-ball.
\newblock {\em Ann. Probab.}, 33(2):480--513, 2005.

\bibitem[BH96]{BH96}
S.~G. Bobkov and C.~Houdr\'e.
\newblock Characterization of {G}aussian measures in terms of the isoperimetric
  property of half-spaces.
\newblock {\em Zap. Nauchn. Sem. S.-Peterburg. Otdel. Mat. Inst. Steklov.
  (POMI)}, 228(Veroyatn. i Stat. 1):31--38, 356, 1996.

\bibitem[BH97]{BH97}
S.~G. Bobkov and C.~Houdr{\'e}.
\newblock Isoperimetric constants for product probability measures.
\newblock {\em Ann. Probab.}, 25(1):184--205, 1997.

\bibitem[BLYZ12]{BLYZ12}
K.~J. B{\"o}r{\"o}czky, E.~Lutwak, D.~Yang, and G.~Zhang.
\newblock The log-{B}runn-{M}inkowski inequality.
\newblock {\em Adv. Math.}, 231(3-4):1974--1997, 2012.

\bibitem[BN02]{BN02}
F.~Barthe and A.~Naor.
\newblock Hyperplane projections of the unit ball of {$l\sp n\sb p$}.
\newblock {\em Discrete Comput. Geom.}, 27(2):215--226, 2002.

\bibitem[BNT16]{BNT16}
K.~Ball, P.~Nayar, and T.~Tkocz.
\newblock A reverse entropy power inequality for log-concave random vectors.
\newblock {\em Studia Math.}, 235(1):17--30, 2016.

\bibitem[CEFM04]{CFM04}
D.~Cordero-Erausquin, M.~Fradelizi, and B.~Maurey.
\newblock The ({B}) conjecture for the {G}aussian measure of dilates of
  symmetric convex sets and related problems.
\newblock {\em J. Funct. Anal.}, 214(2):410--427, 2004.

\bibitem[Das99]{Das99}
S.~Dasgupta.
\newblock Learning mixtures of {G}aussians.
\newblock In {\em 40th {A}nnual {S}ymposium on {F}oundations of {C}omputer
  {S}cience ({N}ew {Y}ork, 1999)}, pages 634--644. IEEE Computer Soc., Los
  Alamitos, CA, 1999.

\bibitem[ENT18]{ENT18}
A.~Eskenazis, P.~Nayar, and T.~Tkocz.
\newblock Sharp comparison of moments and the log-concave moment problem.
\newblock {\em Adv. Math.}, 334:389--416, 2018.

\bibitem[Esk18]{Esk18}
A.~Eskenazis.
\newblock On extremal sections of subspaces of {$L_p$}.
\newblock Preprint available at \url{https://arxiv.org/abs/1806.04333}, 2018.

\bibitem[Fel71]{Fel71}
W.~Feller.
\newblock {\em An introduction to probability theory and its applications.
  {V}ol. {II}.}
\newblock Second edition. John Wiley \& Sons, Inc., New York-London-Sydney,
  1971.

\bibitem[GL10]{GL10}
N.~Gozlan and C.~L{\'e}onard.
\newblock Transport inequalities. {A} survey.
\newblock {\em Markov Process. Related Fields}, 16(4):635--736, 2010.

\bibitem[GNT14]{GNT14}
O.~Gu{\'e}don, P.~Nayar, and T.~Tkocz.
\newblock Concentration inequalities and geometry of convex bodies.
\newblock In {\em Analytical and probabilistic methods in the geometry of
  convex bodies}, volume~2 of {\em IMPAN Lect. Notes}, pages 9--86. Polish
  Acad. Sci. Inst. Math., Warsaw, 2014.

\bibitem[Haa81]{Haa81}
U.~Haagerup.
\newblock The best constants in the {K}hintchine inequality.
\newblock {\em Studia Math.}, 70(3):231--283 (1982), 1981.

\bibitem[HLP88]{HLP88}
G.~H. Hardy, J.~E. Littlewood, and G.~P{\'o}lya.
\newblock {\em Inequalities}.
\newblock Cambridge Mathematical Library. Cambridge University Press,
  Cambridge, 1988.
\newblock Reprint of the 1952 edition.

\bibitem[KMS96]{KMS96}
A.~L. Koldobsky and S.~J. Montgomery-Smith.
\newblock Inequalities of correlation type for symmetric stable random vectors.
\newblock {\em Statist. Probab. Lett.}, 28(1):91--97, 1996.

\bibitem[Kol98]{Kol98}
A.~Koldobsky.
\newblock An application of the {F}ourier transform to sections of star bodies.
\newblock {\em Israel J. Math.}, 106:157--164, 1998.

\bibitem[Kol05]{Kol05}
A.~Koldobsky.
\newblock {\em Fourier analysis in convex geometry}, volume 116 of {\em
  Mathematical Surveys and Monographs}.
\newblock American Mathematical Society, Providence, RI, 2005.

\bibitem[K{\"o}n14]{Koe14}
H.~K{\"o}nig.
\newblock On the best constants in the {K}hintchine inequality for {S}teinhaus
  variables.
\newblock {\em Israel J. Math.}, 203(1):23--57, 2014.

\bibitem[KV07]{KV07}
B.~Klartag and R.~Vershynin.
\newblock Small ball probability and {D}voretzky's theorem.
\newblock {\em Israel J. Math.}, 157:193--207, 2007.

\bibitem[KZ03]{KZ03}
A.~Koldobsky and M.~Zymonopoulou.
\newblock Extremal sections of complex {$l_p$}-balls, {$0<p\le2$}.
\newblock {\em Studia Math.}, 159(2):185--194, 2003.

\bibitem[Lat02]{Lat02}
R.~Lata{\l}a.
\newblock On some inequalities for {G}aussian measures.
\newblock In {\em Proceedings of the {I}nternational {C}ongress of
  {M}athematicians, {V}ol. {II} ({B}eijing, 2002)}, pages 813--822. Higher Ed.
  Press, Beijing, 2002.

\bibitem[LBo14]{LBo14}
A.~Livne Bar-on.
\newblock The ({B}) conjecture for uniform measures in the plane.
\newblock In {\em Geometric aspects of functional analysis}, volume 2116 of
  {\em Lecture Notes in Math.}, pages 341--353. Springer, Cham, 2014.

\bibitem[Lie78]{Lie78}
E.~H. Lieb.
\newblock Proof of an entropy conjecture of {W}ehrl.
\newblock {\em Comm. Math. Phys.}, 62(1):35--41, 1978.

\bibitem[LM17]{LM15}
R.~Lata{\l}a and D.~Matlak.
\newblock Royen's proof of the {G}aussian correlation inequality.
\newblock In {\em Geometric aspects of functional analysis}, volume 2169 of
  {\em Lecture Notes in Math.}, pages 265--275. Springer, Cham, 2017.

\bibitem[LO95]{LO95}
R.~Lata{\l}a and K.~Oleszkiewicz.
\newblock A note on sums of independent uniformly distributed random variables.
\newblock {\em Colloq. Math.}, 68(2):197--206, 1995.

\bibitem[LO05]{LO05}
R.~Lata{\l}a and K.~Oleszkiewicz.
\newblock Small ball probability estimates in terms of widths.
\newblock {\em Studia Math.}, 169(3):305--314, 2005.

\bibitem[LP99]{LP99}
T.~M. Lewis and G.~Pritchard.
\newblock Correlation measures.
\newblock {\em Electron. Comm. Probab.}, 4:77--85, 1999.

\bibitem[LP03]{LP03}
T.~M. Lewis and G.~Pritchard.
\newblock Tail properties of correlation measures.
\newblock {\em J. Theoret. Probab.}, 16(3):771--788, 2003.

\bibitem[Mar16]{Mar16}
A.~Marsiglietti.
\newblock On the improvement of concavity of convex measures.
\newblock {\em Proc. Amer. Math. Soc.}, 144(2):775--786, 2016.

\bibitem[Mem15]{Mem15}
Y.~Memarian.
\newblock On a correlation inequality for {C}auchy type measures.
\newblock {\em New Zealand J. Math.}, 45:53--64, 2015.

\bibitem[MO79]{MO79}
A.~W. Marshall and I.~Olkin.
\newblock {\em Inequalities: theory of majorization and its applications},
  volume 143 of {\em Mathematics in Science and Engineering}.
\newblock Academic Press, Inc. [Harcourt Brace Jovanovich, Publishers], New
  York-London, 1979.

\bibitem[MP65]{MP65}
A.~W. Marshall and F.~Proschan.
\newblock An inequality for convex functions involving majorization.
\newblock {\em J. Math. Anal. Appl.}, 12:87--90, 1965.

\bibitem[MP88]{MP88}
M.~Meyer and A.~Pajor.
\newblock Sections of the unit ball of {$l^n_p$}.
\newblock {\em J. Funct. Anal.}, 80(1):109--123, 1988.

\bibitem[NO12]{NO12}
P.~Nayar and K.~Oleszkiewicz.
\newblock Khinchine type inequalities with optimal constants via ultra
  log-concavity.
\newblock {\em Positivity}, 16(2):359--371, 2012.

\bibitem[NT18]{NT18}
P.~Nayar and T.~Tkocz.
\newblock On a convexity property of sections of the cross-polytope.
\newblock Preprint available at \url{https://arxiv.org/abs/1810.02038}, 2018.

\bibitem[PV18]{PV16a}
G.~Paouris and P.~Valettas.
\newblock A {G}aussian small deviation inequality for convex functions.
\newblock {\em Ann. Probab.}, 46(3):1441--1454, 2018.

\bibitem[PV19]{PV16b}
G.~Paouris and P.~Valettas.
\newblock Variance estimates and almost {E}uclidean structure.
\newblock To appear in {\em Adv. Geom.}, preprint available at
  \url{https://arxiv.org/pdf/1703.10244.pdf}, 2019.

\bibitem[Roy14]{Roy14}
T.~Royen.
\newblock A simple proof of the {G}aussian correlation conjecture extended to
  some multivariate gamma distributions.
\newblock {\em Far East J. Theor. Stat.}, 48(2):139--145, 2014.

\bibitem[RR91]{RR91}
S.~T. Rachev and L.~R{\"u}schendorf.
\newblock Approximate independence of distributions on spheres and their
  stability properties.
\newblock {\em Ann. Probab.}, 19(3):1311--1337, 1991.

\bibitem[Sar15]{Sar15}
C.~Saroglou.
\newblock Remarks on the conjectured log-{B}runn-{M}inkowski inequality.
\newblock {\em Geom. Dedicata}, 177:353--365, 2015.

\bibitem[Sar16]{Sar16}
C.~Saroglou.
\newblock More on logarithmic sums of convex bodies.
\newblock {\em Mathematika}, 62(3):818--841, 2016.

\bibitem[Sim11]{Sim11}
T.~Simon.
\newblock Multiplicative strong unimodality for positive stable laws.
\newblock {\em Proc. Amer. Math. Soc.}, 139(7):2587--2595, 2011.

\bibitem[Sta59]{Sta59}
A.~J. Stam.
\newblock Some inequalities satisfied by the quantities of information of
  {F}isher and {S}hannon.
\newblock {\em Information and Control}, 2:101--112, 1959.

\bibitem[SW49]{SW49}
C.~E. Shannon and W.~Weaver.
\newblock {\em The {M}athematical {T}heory of {C}ommunication}.
\newblock The University of Illinois Press, Urbana, Ill., 1949.

\bibitem[SZ90]{SZ90}
G.~Schechtman and J.~Zinn.
\newblock On the volume of the intersection of two {$L^n_p$} balls.
\newblock {\em Proc. Amer. Math. Soc.}, 110(1):217--224, 1990.

\bibitem[Sza76]{Sza76}
S.~J. Szarek.
\newblock On the best constants in the {K}hinchin inequality.
\newblock {\em Studia Math.}, 58(2):197--208, 1976.

\bibitem[Ver09]{Ver09}
R.~Vershynin.
\newblock Lectures in geometric functional analysis.
\newblock Lecture notes, available at
  \url{https://www.math.uci.edu/~rvershyn/papers/GFA-book.pdf}, 2009.

\bibitem[Wer84]{Wer84}
A.~Weron.
\newblock Stable processes and measures: a survey.
\newblock In {\em Probability theory on vector spaces, {III} ({L}ublin, 1983)},
  volume 1080 of {\em Lecture Notes in Math.}, pages 306--364. Springer,
  Berlin, 1984.

\bibitem[Whi60]{Whi60}
P.~Whittle.
\newblock Bounds for the moments of linear and quadratic forms in independent
  variables.
\newblock {\em Teor. Verojatnost. i Primenen.}, 5:331--335, 1960.

\bibitem[Yu08]{Yu08}
Y.~Yu.
\newblock Letter to the editor: {O}n an inequality of {K}arlin and {R}inott
  concerning weighted sums of i.i.d. random variables.
\newblock {\em Adv. in Appl. Probab.}, 40(4):1223--1226, 2008.

\end{thebibliography}
\bibliographystyle{alpha}
\nocite{*}

\end{document}